\documentclass[11pt,reqno]{amsart}
\usepackage{amssymb}
\usepackage{amsmath}
\usepackage{enumitem} 

\usepackage[dvipsnames]{xcolor}

\usepackage{stmaryrd}

\usepackage{palatino}

\usepackage[
            breaklinks=true,
            pdftitle={},
            pdfauthor={}]{hyperref}
\usepackage{color}
\definecolor{tocolor}{rgb}{.1,.1,.5}
\definecolor{urlcolor}{rgb}{.2,.2,.6}
\definecolor{linkcolor}{rgb}{.0,.3,.7}
\definecolor{citecolor}{rgb}{.7,.1,.3}
\hypersetup{colorlinks=true, urlcolor=urlcolor, linkcolor=linkcolor, citecolor=citecolor}

\input xy
\xyoption{all}

\newtheorem{theorem}{Theorem}[section]
\newtheorem{lemma}[theorem]{Lemma}
\newtheorem{proposition}[theorem]{Proposition}
\newtheorem{corollary}[theorem]{Corollary}

\theoremstyle{definition}
\newtheorem{definition}[theorem]{Definition}
\newtheorem{remark}[theorem]{Remark}

\numberwithin{equation}{section}

\begin{document}

\baselineskip=15pt

\title[Moduli spaces of framed $G$--Higgs bundles]{Moduli spaces of framed $G$--Higgs bundles and 
symplectic geometry}

\author[I. Biswas]{Indranil Biswas}

\address{School of Mathematics, Tata Institute of Fundamental
Research, Homi Bhabha Road, Mumbai 400005, India}

\email{indranil@math.tifr.res.in}

\author[M. Logares]{Marina Logares}

\address{School of Computing Electronics and Mathematics, University of Plymouth, 
Drake Circus, PL4 8AA, Plymouth, United Kingdom} 

\email{marina.logares@plymouth.ac.uk}
 
\author[A. Pe\'on-Nieto]{Ana Pe\'on-Nieto}

\address{Universit\'e de Gen\`eve, Section de Math\'ematiques,
2-4 Rue du Li\`evre, C.P. 64, 1211 Gen\`eve 4, Switzerland}

\email{ana.peon-nieto@unige.ch}



\begin{abstract}
Let $X$ be a compact connected Riemann surface, $D\, \subset\, X$ a reduced effective divisor, 
$G$ a connected complex reductive affine algebraic group and $H_x\, \subsetneq\, G$ a 
Zariski closed subgroup for every $x\, \in\, D$. A framed principal $G$--bundle on $X$ is a pair 
$(E_G,\, \phi)$, where $E_G$ is a holomorphic principal $G$--bundle on $X$ and $\phi$ assigns 
to each $x\, \in\, D$ a point of the quotient space $(E_G)_x/H_x$. A framed $G$--Higgs bundle 
is a framed principal $G$--bundle $(E_G,\, \phi)$ together with a
holomorphic section $\theta\, \in\, 
H^0(X,\, \text{ad}(E_G)\otimes K_X\otimes{\mathcal O}_X(D))$ such that $\theta(x)$ is 
compatible with the framing $\phi$ at $x$ for every $x\, \in\, D$. We construct a holomorphic 
symplectic structure on the moduli space $\mathcal{M}_{FH}(G)$ of stable framed $G$--Higgs 
bundles on $X$. Moreover, we prove that the natural morphism from $\mathcal{M}_{FH}(G)$ to the
moduli space $\mathcal{M}_{H}(G)$ of $D$-twisted $G$--Higgs bundles $(E_G,\, \theta)$ that forgets 
the framing, is Poisson. These results generalize \cite{BLP} where $(G,\, \{H_x\}_{x\in D})$ 
is taken to be $(\text{GL}(r,{\mathbb C}),\, \{\text{I}_{r\times r}\}_{x\in D})$. We also 
investigate the Hitchin system for the moduli space $\mathcal{M}_{FH}(G)$ and its
relationship with that for $\mathcal{M}_{H}(G)$.
\end{abstract}

\maketitle

\tableofcontents

\section{Introduction}

Higgs bundles on Riemann surfaces were introduced by Hitchin in \cite{Hi0} and the Higgs 
bundles on higher dimensional complex manifolds were introduced by Simpson in \cite{Si0}. The 
moduli spaces of Higgs bundles on Riemann surfaces have been extensively studied because of their 
rich symplectic geometric, differential geometric as well as algebraic geometric structures;
they also play an important role in geometric representation theory \cite{Ngo}.
In particular, in his foundational papers \cite{Hi0,H}, Hitchin showed that such a moduli 
space is a holomorphically symplectic manifold which contains the total space of the 
cotangent bundle of a moduli space of vector bundles as a Zariski dense open subset such that 
the restriction of the symplectic form to this Zariski open subset coincides with the 
standard Liouville symplectic form on the total space
of the cotangent bundle. Moreover, he constructed a 
fibration of the moduli space of Higgs bundles over an affine space which he went on to prove 
to be an algebraically completely integrable system; this completely integrable system 
nowadays is known as the Hitchin system.

Over time, moduli spaces of Higgs bundles have undergone diverse generalizations. Here we 
will consider $D$--twisted $G$--Higgs bundles to which we shall add an extra structure which 
is called a framing. Similar objects were considered earlier in \cite{Si1}, \cite{Si2}, 
\cite{Ma} and \cite{Nitin}.

Take a compact connected Riemann surface $X$, and fix a reduced effective divisor $D$ on it. 
Let $G$ be a connected reductive affine algebraic group defined over $\mathbb C$. A 
$D$--twisted $G$--Higgs bundle $(E_{G},\,\theta)$ on $X$ consists of a holomorphic principal 
$G$--bundle $E_{G}\,\longrightarrow\, X$ together with a $D$-twisted Higgs field 
$\theta\,\in\, H^{0}(X, {\rm ad}(E_{G})\otimes K_{X}\otimes{\mathcal O}_X(D))$, where ${\rm 
ad}(E_{G})$ is the adjoint vector bundle for principal $G$--bundle $E_{G}$ while $K_{X}$ 
denotes the holomorphic cotangent bundle of $X$.

The isomorphism classes of all topological principal $G$--bundles on $X$ are parametrized by 
the fundamental group $\pi_{1}(G)$. Once we fix a topological isomorphism class 
$\nu\,\in\,\pi_{1}(G)$, the moduli space of stable $D$--twisted $G$--Higgs bundles is a 
smooth connected orbifold \cite{Nitin,H2}; such a moduli space will be denoted by 
$\mathcal{M}_{H}(G)$.

It is known that $\mathcal{M}_{H}(G)$ is equipped with a natural holomorphic Poisson
structure \cite{Bot,Ma,BR}. It should be mentioned that this Poisson
structure is never symplectic unless $D$ is actually the zero divisor.

Fix a nondegenerate symmetric $G$--invariant bilinear form $\sigma$ on
the Lie algebra ${\mathfrak g}\,:=\, 
\text{Lie}(G)$. For each point $x\, \in\, D$, fix a Zariski closed subgroup $H_x\, \subset\, 
G$. A framing on a holomorphic principal $G$--bundle $E_G$ on $X$ is a map $\phi\, :\, D\, 
\longrightarrow\,\bigcup_{x\in D} (E_G)_x/H_x$ such that $\phi(x)\, \in\, (E_G)_x/H_x$
for every $x\, \in\, D$. Using the bilinear form $\sigma$ in (\ref{e5}) and the framing
$\phi$ on $E_G$, we construct a subspace ${\mathcal H}^\perp_x\, \subset\, \text{ad}(E_G)_x$ 
for each $x\, \in\, D$; see Section \ref{se2.1} for the construction. Let
$$
\text{ad}^n_\phi(E_G)\, \subset\, \text{ad}(E_G)
$$
be the subsheaf uniquely identified by the condition that a locally defined holomorphic section $s$ of
$\text{ad}(E_G)$ lies in $\text{ad}^n_\phi(E_G)$ if and only if $s(x)\, \in\,
{\mathcal H}^\perp_x\, \subset\, \text{ad}(E_G)_x$ for every $x\,\in\, D$ that lies in the
domain of the locally defined section $s$.

A Higgs field on a framed principal $G$--bundle $(E_G,\, \phi)$ is a holomorphic section 
$\theta$ of the holomorphic vector bundle $\text{ad}^n_\phi(E_G)\otimes K_X\otimes {\mathcal 
O}_X(D)$, where $\text{ad}^n_\phi(E_G)$ is described above. Such a triple $(E_G,\, \phi,\, 
\theta)$ will be called a framed $G$--Higgs bundle. In particular, the pair $(E_G,\, \theta)$ 
is a $D$--twisted $G$--Higgs bundle. If $H_x$ is the trivial subgroup $e\,\in\, G$ for all 
$x\,\in\, D$, then ${\mathcal H}^\perp_x\,=\, \text{ad}(E_G)_x$ for all $x$. Hence in that 
case a Higgs field on $(E_G,\, \phi)$ is simply an element of $H^{0}(X, \,{\rm 
ad}(E_{G})\otimes K_{X}\otimes{\mathcal O}_X(D))$ (a $D$--twisted $G$--Higgs field on $E_G$).

We prove the following:
\begin{enumerate}
\item \textit{A moduli space of framed $G$--Higgs bundles has a natural holomorphic symplectic
structure.} (See Theorem \ref{thm1}.)

\item \textit{The forgetful map from a moduli space of framed $G$--Higgs bundles to a
moduli space of $D$--twisted $G$--Higgs bundles, defined by $(E_G,\, \phi,\, \theta)\,
\longmapsto\, (E_G,\, \theta)$, is Poisson.} (See Theorem \ref{thm2}.)
\end{enumerate}

In particular, the holomorphic Poisson manifold given by a moduli space of $D$--twisted
$G$--Higgs bundles $\mathcal{M}_{H}(G)$ can be enhanced to a symplectic manifold by augmenting
the $D$--twisted $G$--Higgs bundle with a framing for the trivial sub group
$e\,\in\, G$ for all points of $D$.

The Hitchin system
\begin{equation}\label{eq:hmf_intro}
h\,:\,\mathcal{M}_H(G)\,\longrightarrow\, \mathcal{B}
\end{equation}
is defined by evaluating the Chevalley morphism 
$\chi\,:\,\mathfrak{g}\,\longrightarrow\, \mathfrak{g}\,\sslash G$
on the Higgs field. This
is again, despite $\mathcal{M}_H(G)$ being only Poisson and not symplectic unless $D$ is the zero divisor, an algebraically completely integrable system
(\cite[Remark 8.6]{Ma}, \cite[Section 5]{DMa}).

Hitchin systems constitute a very large family of algebraically completely integrable 
systems. Moreover, it is known that for suitable choices of the Riemann surface $X$, the 
group $G$, and the twisting, many classical integrable systems are embedded in them as 
symplectic leaves (see \cite[Section 9]{Ma}). In \cite{BLP} we showed that the Hitchin systems 
provided by the moduli spaces of framed $G$--Higgs bundles, when $G\,=\,\text{GL}(r,\mathbb{C})$ 
and $H_{x}\,=\,{\rm I}_{r\times r}$ for all $x\in D$, are no longer algebraically completely 
integrable systems. Firstly, the number of Poisson commuting functions given by the Hitchin map 
falls short of the dimension of the moduli space of framed principal $G$--bundles
(which is half the dimension of the moduli spaces of framed $G$--Higgs bundles).
Secondly, its
fibers are not abelian varieties, but torsors over the fibers of the
non-framed Hitchin system. We also investigate
two subsystems which come with the correct number of Poisson commuting functions. We also
show that these results generalize for any connected complex reductive affine algebraic group
$G$.

Let $\mathcal{M}_{FH}(G)$ denote the moduli space of stable framed $G$--Higgs bundles with a fixed topological class
$\nu$, and let
\begin{equation}\label{eq:hfG_intro}
h_{FH}\,:\,\mathcal{M}_{FH}(G)\,\longrightarrow \,\mathcal{B}
\end{equation}
be the corresponding Hitchin system.

We prove the following:

\begin{enumerate}[resume]
\item {\it The Hitchin map $h_{FH}$ in \eqref{eq:hfG_intro} produces a set of $N\,:=\,
\dim\,\mathcal{B}$ Poisson commuting functions on 
$\mathcal{M}_{FH}(G)$, i.e., $h_{FH}\,=\,(h_{1},\,\cdots,\, h_{N})$ with $\{h_{i},\,h_{j}\}_{P}
\,=\,0$ for all $i,\,j\,=\,1,\, \cdots ,\,N$.} (See Corollary 
\ref{cor:HF-Poissonc}.)

\item {\it The generic fibers of the map $h_{FH}$ are torsors over the abelian varieties $J_{b}=h^{-1}(b)$ where 
$h\,:\,\mathcal{M}_{H}(G)\,\longrightarrow\, \mathcal B$} is the Hitchin map in \eqref{eq:hmf_intro}. (See Corollary \ref{cor:torsor-fibres}.)

\item {\it There is a moduli space $\mathcal{M}_{FH}^\Delta(G)$ which is a subsystem of $\mathcal{M}_{FH}(G)$
and it is maximally abelianizable.} (See Corollary \ref{cor:subsystem} and Remark \ref{rk:max_abelianizable_system}.)
\end{enumerate}

The above results 
specialize to the results in \cite{BLP} when $G\,=\,\text{GL}(r,\mathbb{C})$ and
$H_{x}\,=\,{\rm I}_{r\times r}$ for every $x\,\in\, D$.

In Section \ref{sec:framedG} we introduce $D$-twisted $G$--Higgs bundles as well as framed 
structures for principal bundles and their juxtaposition, namely framed $G$--Higgs bundles. In
Section \ref{se3} we study 
the infinitesimal deformations of the $D$-twisted $G$--Higgs bundles and framed 
$G$--Higgs bundles. In Section \ref{sec:SP}, we construct a symplectic structure on the 
moduli space $\mathcal{M}_{FH}(G)$ of stable framed $G$--Higgs bundles, as well as a Poisson 
structure on the moduli space $\mathcal{M}_{H}(G)$ of stable $D$-twisted $G$--Higgs bundles.

In Section \ref{sec:cameral} we investigate the integrability properties of the Hitchin 
system in \eqref{eq:hmf_intro}. For the sake of clarity, we focus on the case $H_x\,=\,e$ for 
all $x$, nevertheless discussing the general case in Remark 
\ref{rk:max_abelianizable_system}.

We also describe a subsystem of \eqref{eq:hfG_intro} which is maximally abelianizable. This 
is done using the cameral cover approach of Donagi--Gaitsgory \cite{DG} (see also 
\cite{Ngo}), which identifies the generic fiber of the Hitchin map in \eqref{eq:hmf_intro} 
with a subvariety of the Jacobian of the cameral cover. We find that the generic fibers 
(corresponding to smooth cameral covers unramified over $D$) are $G^n/Z(G)$--torsors over the 
fibers of the map in \eqref{eq:hmf_intro}, where $n\,=\, \#D$ and $Z(G)$ is the center of 
$G$. It turns out that we may naturally identify $T^n/Z(G)$--sub-torsors therein (where 
$T\subset G$ is a maximal torus) with moduli spaces of framed Higgs bundles. More precisely, 
they correspond to the fibers of the restriction of the Hitchin map to the locus of 
relatively framed Higgs bundles defined in \eqref{eq:diag_FH}. This parametrizes Higgs 
bundles together with a framing of both the bundle and the Higgs field (see Proposition 
\ref{prop_relative Picard}, Theorem \ref{thm:fibers} and Remark \ref{rk:relatively_framed}).

\section{Framed $G$--Higgs bundles and stability}\label{sec:framedG}

\subsection{Framings and $G$--Higgs bundles}\label{se2.1}

Let $X$ be a compact connected Riemann surface. Denote by $K_{X}$ the holomorphic cotangent
bundle of $X$. Let
 \begin{equation}\label{e1}
D\,=\,\{x_{1},\,\cdots,\, x_{n}\} \,\subset\, X 
\end{equation}
be a reduced effective divisor on $X$ consisting of $n\, \geq\, 1$ distinct
points.

To clarify, we shall always assume that $D\, \not=\, \emptyset$.

The holomorphic line bundle $K_X\otimes {\mathcal O}_X(D)$ on $X$ will be denoted by 
$K_X(D)$. For any $x\, \in\, D$, the fiber $K_X(D)_x$ of $K_X(D)$ over $x$ is identified with 
$\mathbb C$. Indeed, for any holomorphic coordinate function $z$ on $X$ defined around the 
point $x$ such that $z(x)\,=\, 0$, consider the homomorphism
\begin{equation}\label{res}
{\mathbb C}\, \longrightarrow\, K_X(D)_x\, ,
\ \ c\, \longmapsto\, c\cdot \frac{dz}{z}\Big\vert_{z=x}\, .
\end{equation}
The homomorphism in \eqref{res} is in fact independent of the choice of the above holomorphic
coordinate function $z$, and thus $K_X(D)_x$ is canonically identified with $\mathbb C$.

Let $G$ be a connected complex Lie group. Let
\begin{equation}\label{e2}
p\, :\, E_G\, \longrightarrow\, X
\end{equation}
be a holomorphic principal $G$--bundle over $X$; we recall that this means that $E_G$ is a
holomorphic fiber bundle over $X$ equipped with a holomorphic right-action of the group $G$
$$
q'\, :\, E_G\times G\,\longrightarrow\, E_G
$$
such that
\begin{equation}\label{qp}
p(q'(z,\, g))\,=\, p(z)
\end{equation}
for all $(z,\, g)\, \in\, E_G\times G$, where
$p$ is the projection in \eqref{e2} and, furthermore, the resulting
map to the fiber product
$$
E_G\times G \, \longrightarrow\, E_G\times_X E_G\, , \ \ (z,\, g) \, \longrightarrow\, (z,\,
q'(z,\, g))
$$
is a biholomorphism. For notational convenience, the point
$q'(z,\, g)\, \in\, E_G$, where $(z,\, g)\,\in\, E_G\times G$, will be denoted by $zg$. For any
$x\, \in\, X$, the fiber $p^{-1}(x)\, \subset\, E_G$ will be denoted by $(E_G)_x$.

For each point $x\, \in\, D$, fix a complex Lie proper subgroup
\begin{equation}\label{hx}
H_x\, \subsetneq\, G\, .
\end{equation}

A \textit{framing} of $E_G$ over the divisor $D$ in \eqref{e1} is a map
$$
\phi\, :\, D\, \longrightarrow\, \bigcup_{x\in D} (E_G)_x/H_x\, ,
$$
where $H_x$ is the subgroup in \eqref{hx},
such that $\phi(x) \, \in\, (E_G)_x/H_x$ for every $x\, \in\, D$. So
the space of all framings of $E_G$ over $D$ is the Cartesian product
\begin{equation}\label{e3}
{\mathcal F}(E_G)\, :=\, \prod_{x\in D} (E_G)_x/H_x\, .
\end{equation}
Let
\begin{equation}\label{e4}
\widehat{p}_x\, :\, {\mathcal F}(E_G)\, \longrightarrow\, (E_G)_x/H_x
\end{equation}
be the natural projection.

A \textit{framed} principal $G$--bundle on $X$ is a holomorphic
principal $G$--bundle $E_G$ on $X$ equipped with a framing over $D$.

The first remark below is due to the referee.

\begin{remark}\label{remP}
Take $G$ to be a reductive algebraic group.
A parabolic subgroup of $G$ is a Zariski closed connected subgroup $P\, \subset\, G$ such that
the quotient variety $G/P$ is projective.
Set each $H_x$ to be some parabolic subgroup 
of $G$. Then a framed principal $G$--bundle is a quasiparabolic $G$--bundle with parabolic
divisor $D$ and quasiparabolic type $H_x$ for $x\, \in\, D$. In particular, when
$G\,=\, \text{GL}(r,{\mathbb C})$ and $H_x\,\subset\, \text{GL}(r,{\mathbb C})$ is a parabolic
subgroup for every $x\, \in\, D$, a framed principal $G$--bundle corresponds to
a holomorphic vector bundle $E$ on $X$ of rank $r$ equipped with a strictly
decreasing filtration, by linear subspaces, of the fiber $E_x$ for all $x\, \in\, D$. The
dimensions of the subspaces in the filtration of $E_x$ are determined by the conjugacy
class of the subgroup $H_x$. Conversely, these dimensions determine the conjugacy class of $H_x$.
\end{remark}

\begin{remark}\label{rem0}
If $H_x$ is a normal subgroup of $G$, then the action of $G$ on $(E_G)_x$ produces an
action of the quotient group $G/H_x$ on the quotient manifold $(E_G)_x/H_x$. This
action of $G/H_x$ on $(E_G)_x/H_x$ is evidently free and transitive. In other words,
$(E_G)_x/H_x$ is a torsor for the group
$G/H_x$. Therefore, if $H_x$ is a normal subgroup
of $G$ for all $x\, \in\, D$, then ${\mathcal F}(E_G)$ in \eqref{e3} is a torsor for the
group $\prod_{x\in D} G/H_x$. The special case where 
$(G,\, \{H_x\}_{x\in D})\,=\, (\text{GL}(r,{\mathbb C}),\, \{\text{I}_{r\times r}\}_{x\in D})$
is treated in \cite{BLP}.
\end{remark}

Let $\mathfrak g$ denote the Lie algebra of $G$. For any $x\, \in\, D$, the Lie algebra of the
subgroup $H_x$ in \eqref{hx}
will be denoted by ${\mathfrak h}_x$. Since the adjoint action of $H_x$ on $\mathfrak g$
preserves the sub-algebra ${\mathfrak h}_x$, the quotient space ${\mathfrak g}/{\mathfrak h}_x$
is equipped with an action of $H_x$ induced by the adjoint action of $H_x$.

The quotient map $(E_G)_x\, \longrightarrow\, (E_G)_x/H_x$ defines a holomorphic principal
$H_x$--bundle over the complex manifold $(E_G)_x/H_x$. Let
$$
V^0_x\,:=\, (E_G)_x\times^{H_x} ({\mathfrak g}/{\mathfrak h}_x)\, \longrightarrow\, (E_G)_x/H_x
$$
be the holomorphic vector bundle over $(E_G)_x/H_x$ associated to this holomorphic principal $H_x$--bundle
$(E_G)_x\, \longrightarrow\, (E_G)_x/H_x$ for the above $H_x$--module
${\mathfrak g}/{\mathfrak h}_x$. Then the holomorphic tangent bundle
of the space of all framings ${\mathcal F}(E_G)$ defined in \eqref{e3} has the expression
\begin{equation}\label{tp}
T{\mathcal F}(E_G)\,=\, \bigoplus_{x\in D} \widehat{p}^*_x V^0_x\, ,
\end{equation}
where $\widehat{p}_x$ is the projection in \eqref{e4}.

Henceforth, $G$ will always be assumed to be a connected complex reductive affine algebraic group.
The subgroup $H_x\, \, \subset\, G$ in \eqref{hx} is assumed to be Zariski closed for every $x\, \in\, D$.

Since the group $G$ is reductive, its Lie algebra $\mathfrak g$ admits $G$--invariant nondegenerate
symmetric bilinear forms. To construct such a form, consider the decomposition ${\mathfrak g}\,=\,
Z({\mathfrak g})\oplus [{\mathfrak g},\, {\mathfrak g}]$, where $Z({\mathfrak g})$ is the center of
${\mathfrak g}$. Take the Killing form $\kappa$ on $[{\mathfrak g},\, {\mathfrak g}]$ and take
any nondegenerate symmetric bilinear form $\sigma'$ on $Z({\mathfrak g})$; the direct sum
$\sigma'\oplus \kappa$ is a $G$--invariant nondegenerate
symmetric bilinear form on $Z({\mathfrak g})\oplus [{\mathfrak g},\, {\mathfrak g}] \,=\,
\mathfrak g$. Fix a $G$--invariant nondegenerate symmetric bilinear form
\begin{equation}\label{e5}
\sigma\, :\, \text{Sym}^2({\mathfrak g})\, \longrightarrow\,\mathbb C
\end{equation}
on $\mathfrak g$.

Take a holomorphic principal $G$--bundle $E_G$ on $X$. Let $\text{ad}(E_G)$ be the adjoint vector
bundle over $X$ associated to $E_G$ for the adjoint action of $G$ on $\mathfrak g$. Therefore, each
fiber of $\text{ad}(E_G)$ is a Lie algebra isomorphic to $\mathfrak g$. More precisely,
for any $y\, \in\, X$, there is an isomorphism of Lie algebras
${\mathfrak g}\, \stackrel{\sim}{\longrightarrow}\,
\text{ad}(E_G)_y$ which is unique up to automorphisms of $\mathfrak g$ given by the adjoint action
of the elements of $G$. Using such an
isomorphism ${\mathfrak g}\, \longrightarrow\,
\text{ad}(E_G)_y$, the $G$--invariant form $\sigma$ in \eqref{e5}
produces a symmetric nondegenerate bilinear form on the fiber $\text{ad}(E_G)_y$; note that
this bilinear form on $\text{ad}(E_G)_y$ does not depend on the choice of the above isomorphism
${\mathfrak g}\, \longrightarrow\, \text{ad}(E_G)_y$ because $\sigma$ is $G$--invariant. Let
\begin{equation}\label{e6}
\widehat{\sigma}\, :\, \text{Sym}^2({\rm ad}(E_G))\, \longrightarrow\,{\mathcal O}_X
\end{equation}
be the bilinear form constructed as above using $\sigma$. Let
$$
T^p_{\rm rel}\, \subset\, TE_G
$$
be the relative tangent bundle for the projection $p$ in \eqref{e2}.
The action of $G$ on $E_G$ produces an action of $G$ on $TE_G$. This action preserves
the subbundle $T^p_{\rm rel}$ because of the condition
in \eqref{qp}. The trivial holomorphic vector bundle
$E_G\times{\mathfrak g}\, \longrightarrow\, E_G$ equipped with the action of $G$, given by the
action of $G$ on $E_G$ and the adjoint action of $G$ on $\mathfrak g$, is identified with
$T^p_{\rm rel}$; this identification between $T^p_{\rm rel}$ and
$E_G\times{\mathfrak g}$ is evidently $G$--equivariant. The quotient $T^p_{\rm rel}/G$
is a holomorphic vector bundle over $E_G/G\,=\, X$. This holomorphic vector bundle
$T^p_{\rm rel}/G$ over $X$ is holomorphically identified with
the adjoint vector bundle $\text{ad}(E_G)$.

Let $\phi\, :\, D\, \longrightarrow\, \bigcup_{x\in D} (E_G)_x/H_x$ be a framing on $E_G$. For
each $x\, \in\, D$, let
\begin{equation}\label{qx}
q_x\,:\, (E_G)_x\,\longrightarrow\, (E_G)_x/H_x
\end{equation}
be the natural quotient map.

Using the framing $\phi$ we shall construct a subspace ${\mathcal H}_x\, \subset\,\text{ad}(E_G)_x$ for 
every $x\, \in\, D$. For that purpose, first recall that the elements of $\text{ad}(E_G)_x$ are the $G$--invariant 
sections of the vector bundle $T^p_{\rm rel}\vert_{p^{-1}(x)}\, \longrightarrow\, p^{-1}(x)$, 
where $p$ is the projection in \eqref{e2}. Consider all $G$--invariant
sections $$v\, \in\, H^0(p^{-1}(x),\, T^p_{\rm rel}\vert_{p^{-1}(x)})^G$$ such that
the restriction $v\vert_{q^{-1}_x(\phi(x))}$ satisfies the condition that
$$
v\vert_{q^{-1}_x(\phi(x))}\,\subset\, q^{-1}_x(\phi(x))\times {\mathfrak h}_x\,\subset\, q^{-1}_x
(\phi(x))\times {\mathfrak g}\,=\, T^p_{\rm rel}\vert_{q^{-1}_x(\phi(x))}\, ,
$$
where $q_x$ is the projection in \eqref{qx}, and ${\mathfrak h}_x$
as before is the Lie algebra of $H_x$; here we are using
the earlier observation that $T^p_{\rm rel}\,=\, E_G\times{\mathfrak g}$, and we also have
identified the section $v\vert_{q^{-1}_x(\phi(x))}$ with the subset of
$T^p_{\rm rel}\vert_{q^{-1}_x(\phi(x))}$ given by its image. Let
\begin{equation}\label{e7}
{\mathcal H}_x\, \subset\,\text{ad}(E_G)_x
\end{equation}
be the subspace defined by all such $v$. Note that ${\mathcal H}_x$ is a Lie subalgebra of
$\text{ad}(E_G)_x$ which is identified with ${\mathfrak h}_x$ by an isomorphism
that is unique up to
automorphisms of ${\mathfrak h}_x$ given by the adjoint action of the elements of the group $H_x$.

The following construction of ${\mathcal H}_x$ was suggested by the referee.

\begin{remark}\label{remH}
The framing $\phi$ produces a reduction of structure group of the principal
$G$--bundle $(E_G)_x\, \longrightarrow\, x$, defined just over the point $x$, to the subgroup
$H_x\, \subset\, G$ for each $x\, \in\, D$. Indeed,
$$
E^x_{H_x}\, :=\, q^{-1}_x(\phi(x))\, \subset\, (E_G)_x
$$
is a principal $H_x$--bundle, where $q_x$ and $\phi$ are the maps in
\eqref{qx} and \eqref{e2} respectively. So we have
$$
\text{ad}(E^x_{H_x})\, \subset\, \text{ad}((E_G)_x)\,=\, \text{ad}(E_G)_x\, .
$$
The subspace ${\mathcal H}_{x}$ in \eqref{e7} coincides with $\text{ad}(E^x_{H_x})$.
\end{remark}

For every $x\, \in\, D$, let
\begin{equation}\label{f1}
{\mathcal H}^\perp_x\, \subset\,\text{ad}(E_G)_x
\end{equation}
be the annihilator of
${\mathcal H}_x$ with respect to the bilinear form $\widehat{\sigma}(x)$ constructed in \eqref{e6}.

A \textit{Higgs field} on the framed principal $G$--bundle $(E_G,\, \phi)$ is a holomorphic section
$$
\theta\, \in\, H^0(X,\, \text{ad}(E_G)\otimes K_X(D))
$$
such that $$\theta(x)\, \in\, {\mathcal H}^\perp_x\, \subset\,
\text{ad}(E_G)_x$$ for every $x\, \in\, D$; recall from \eqref{res}
that $K_X(D)_x\,=\, \mathbb C$, so we have $(\text{ad}(E_G)\otimes K_X(D))_x\, =\, \text{ad}(E_G)_x$,
and hence we have $\theta(x)\, \in\, \text{ad}(E_G)_x$.

Notice that in \cite{BLP} the interlinking between the framing and the Higgs field was not
explicit due to the assumption that $H_{x}\,=\,e$ for all $x\,\in\, D$.

\begin{definition}\label{def1}
A \textit{framed} $G$--\textit{Higgs bundle} is a triple of the form $(E_G,\, \phi,\, \theta)$,
where $(E_G,\, \phi)$ is a framed principal $G$--bundle on $X$, and $\theta$ is a Higgs field
on $(E_G,\, \phi)$.
\end{definition}

The following remark is due to the referee.

\begin{remark}\label{remP2}
As in Remark \ref{remP}, assume that each $H_x$ is a parabolic subgroup of $G$.
Therefore, a framed principal $G$--bundle $(E_G,\, \phi)$ is also a quasiparabolic $G$--bundle.
Then a Higgs field on $(E_G,\, \phi)$ is a logarithmic Higgs field $\theta$ on $E_G$,
with polar part on $D$, such that the residue of $\theta$ at every $x\, \in\, D$
is nilpotent with respect to the quasiparabolic structure at $x$. Recall from
Remark \ref{remP} that when $G\,=\, \text{GL}(r,{\mathbb C})$ and $H_x$ is a parabolic
subgroup for all $x\, \in\, D$, a framed principal $G$--bundle
$(E_G,\, \phi)$ corresponds to a holomorphic vector bundle $E$ on $X$ of rank $r$ equipped
with a filtration of subspaces of $E_x$ for every $x\, \in\, D$. In that case, a Higgs field
on $(E_G,\, \phi)$ is a strongly parabolic Higgs field on the quasiparabolic
bundle $E$; see \cite{LM} for strongly versus non-strongly parabolic Higgs fields.
\end{remark}

\subsection{Stability of framed $G$--Higgs bundles}

Recall that a parabolic subgroup of $G$ is a Zariski closed connected subgroup $P$ such that
the quotient variety $G/P$ is projective.
Let $Z_0(G)$ denote the (unique) maximal connected subgroup of the center of $G$. A character
$$\widehat{\chi}\, :\, P\, \longrightarrow\, {\mathbb C}^*\,=\, {\mathbb C}\setminus\{0\}$$
of a parabolic subgroup $P\, \subset\, G$
is called \textit{strictly anti-dominant} if
\begin{itemize}
\item $\widehat{\chi}$ is trivial on $Z_0(G)$ (note that $Z_0(G)\, \subset\, P$), and

\item the holomorphic line bundle over $G/P$ associated to the holomorphic principal $P$--bundle
$G\, \longrightarrow\, G/P$ for the character $\widehat{\chi}$ of $P$ is ample.
\end{itemize}

The unipotent radical of a parabolic subgroup $P\, \subset\, G$ is denoted by
$R_u(P)$. The quotient $P/R_u(P)$ is a reductive affine complex algebraic group. A Zariski closed connected reductive
complex algebraic subgroup $L(P)\, \subset\, P$ is called a Levi factor of $P$ if the composition
of maps
$$
L(P)\, \hookrightarrow\, P\, \longrightarrow\, P/R_u(P)
$$
is an isomorphism \cite[p. 158, \S~11.22]{Bor}. There are Levi factors of $P$, moreover, any two Levi factors
of $P$ differ by the inner automorphism of $P$ produced by an element of the unipotent radical
$R_u(P)$ \cite[p. 158, \S~11.23]{Bor}, \cite[\S~30.2, p. 184]{Hu}.

Let $E_G$ be a holomorphic principal $G$--bundle over $X$. Let 
$$
\theta\, \in\, H^0(X,\, \text{ad}(E_G)\otimes K_X(D))
$$
be a holomorphic section. The $D$--twisted $G$--Higgs bundle $(E_G,\, \theta)$ is called \textit{stable}
(respectively, \textit{semistable}) if for all triples of the form $(P,\, E_P,\, \widehat{\chi})$, where
\begin{itemize}
\item{} $P\, \subset\, G$ is a proper (not necessarily maximal) parabolic subgroup,

\item $E_P\, \subset\, E_G$ is a holomorphic reduction of structure group of $E_G$ to
$P$ over $X$ such that
$$
\theta\, \in\, H^0(X,\, \text{ad}(E_P)\otimes K_X(D))\, \subset\, H^0(X,\, \text{ad}(E_G)\otimes K_X(D))\, ,
$$
and

\item $\widehat{\chi}$ is a strictly anti-dominant character of $P$,
\end{itemize}
the inequality
$$
\text{degree}(E_P(\widehat{\chi}))\, > \, 0
$$
(respectively, $\text{degree}(E_P(\widehat{\chi}))\, \geq \, 0$) holds, where $E_P(\widehat{\chi})$ is the
holomorphic line bundle over $X$ associated to the holomorphic principal $P$--bundle $E_P$ for the
character $\widehat{\chi}$ of $P$. (See \cite{H}, \cite{Si1}, \cite{Si2}, \cite{Ra2}, \cite{Ra}, \cite{RS},
\cite{AB}, \cite{BG}.)

Let $P$ be a parabolic subgroup of $G$ and $E_P\, \subset\, E_G$
a holomorphic reduction of structure group of $E_G$ over $X$ to the subgroup $P$. Such a reduction
of structure group is called \textit{admissible} if for every character $\widehat{\chi}$ of $P$ trivial
on $Z_0(G)$, the associated holomorphic line bundle $E_P(\widehat{\chi})$ on $X$ is of
degree zero.

A $D$--twisted $G$--Higgs bundle $(E_G,\, \theta)$ is called \textit{polystable} if either $E_G$
is stable, or there is a parabolic subgroup $P\, \subset\, G$ and a holomorphic reduction of
structure group $E_{L(P)}\, \subset\, E_G$ over $X$ to a Levi factor $L(P)$ of $P$, such that
\begin{itemize}
\item $\theta\, \in\, H^0(X,\, \text{ad}(E_{L(P)})\otimes K_X(D))\, \subset\,
H^0(X,\, \text{ad}(E_G)\otimes K_X(D))$,

\item the holomorphic $L(P)$--Higgs bundle $(E_{L(P)},\, \theta)$ is stable, and

\item the reduction of structure group of $E_G$ to $P$ given by the extension of the structure
group of $E_{L(P)}$ to $P$, corresponding to the inclusion of $L(P)$ in $P$, is admissible.
\end{itemize}
(See \cite{H}, \cite{Si1}, \cite{Si2}, \cite{RS}, \cite{AB}, \cite{BG}.) In particular, a
polystable $D$--twisted $G$--Higgs bundle is semistable.

\begin{definition}\label{def2}
A framed $G$--Higgs bundle $(E_G,\, \phi,\, \theta)$ over $X$ is called \textit{stable}
if the $D$--twisted $G$--Higgs bundle
$(E_G,\, \theta)$ is stable. Similarly, $(E_G,\, \phi,\, \theta)$ is called
\textit{semistable} (respectively, \textit{polystable}) if $(E_G,\, \theta)$ is semistable
(respectively, polystable).
\end{definition}

It should be mentioned that there are other definitions of (semi)stability of a
framed $G$--Higgs bundle. The one given in Definition \ref{def2} is in fact a special case.

\begin{remark}\label{stability-framed}
When $G\,=\,\text{GL}(r,\mathbb{C})$ and 
$\text{H}_{x}\,=\,\text{I}_{r\times r}$ for all $x\,\in\, D$, Definition \ref{def2} reduces to
the definition of (semi)stable framed Higgs bundles given in \cite[Definition 2.2]{BLP} (see
also \cite[Remark 2.6]{BLP}).
\end{remark}

\section{Infinitesimal deformations}\label{se3}

\subsection{Infinitesimal deformations of a framed principal bundle}\label{se3.1}

The infinitesimal deformations of a holomorphic principal $G$--bundle $E_G$ over
$X$ are parametrized by $H^1(X,\, \text{ad}(E_G))$ (see \cite{Do}, \cite[Appendix III]{Se}).

To describe the space of all infinitesimal deformations of a framed holomorphic principal $G$--bundle,
first consider the special case where $H_x\, =\, \{e\}$ for every $x\, \in\, D$ (see \eqref{hx}); as
before the identity element of $G$ is denoted by $e$. In this case, the infinitesimal deformations
of a framed holomorphic principal $G$--bundle $(E_G,\, \phi)$ are parametrized by $H^1(X,\,
\text{ad}(E_G)\otimes {\mathcal O}_X(-D))$; in the
special case where $G\,=\, \text{GL}(r,{\mathbb C})$ and $H_x\, =\,
\{{\rm I}_{r\times r}\}$ for all $x\,\in\, D$, this is Lemma 2.5 of \cite{BLP}. For notational convenience, the
tensor product $\text{ad}(E_G)\otimes {\mathcal O}_X(-D)$ will be denoted by $\text{ad}(E_G)(-D)$.
Consider the following short exact sequence of coherent analytic sheaves on $X$:
$$
0\, \longrightarrow\, \text{ad}(E_G)(-D)\, \longrightarrow\, \text{ad}(E_G) \, \longrightarrow\, \text{ad}(E_G)\vert_D
\, \longrightarrow\, 0\, .
$$
Let
\begin{equation}\label{e8}
\, \longrightarrow\, H^0(X,\, \text{ad}(E_G)) \, \longrightarrow\,
H^0(X,\, \text{ad}(E_G)\vert_D) \, \stackrel{\alpha_1}{\longrightarrow}\,
H^1(X,\, \text{ad}(E_G)(-D))
\end{equation}
$$
\stackrel{\alpha_2}{\longrightarrow}\, H^1(X,\, \text{ad}(E_G)) \, \longrightarrow\,
H^1(X,\, \text{ad}(E_G)\vert_D)\,=\, 0
$$
be the long exact sequence of cohomologies associated to it; we have
$H^1(X,\, \text{ad}(E_G)\vert_D)\,=\, 0$ in \eqref{e8}
because $\text{ad}(E_G)\vert_D$ is a torsion
sheaf supported on points. The homomorphism $\alpha_2$ in \eqref{e8}
sends an infinitesimal deformation of $(E_G,\, \phi)$ to the infinitesimal deformation
of $E_G$ obtained from it by simply forgetting the framing. Consider the space of framings
${\mathcal F}(E_G)$ in \eqref{e3} (at present $H_x\, =\, \{e\}$ for every $x\, \in\, D$). Note that
$$
T_{\phi}{\mathcal F}(E_G)\,=\, H^0(X,\, \text{ad}(E_G)\vert_D) \,=\, {\mathfrak g}^D\,:=\,
\bigoplus_{x\in D} \mathfrak g\, .
$$
Indeed, $\phi(x)\, \in\, (E_G)_x$ identifies the fiber $(E_G)_x$ with $G$ by sending
any $g\, \in\, G$ to $\phi(x)g \, \in\, (E_G)_x$. This trivialization of $(E_G)_x$
produces an identification of the Lie algebra $\text{ad}(E_G)_x$ with $\mathfrak g$; indeed,
both $\text{ad}(E_G)_x$ and $\mathfrak g$
are identified with the right $G$--invariant vector fields on $(E_G)_x$ and $G$ respectively.
The homomorphism $\alpha_1$ in \eqref{e8} gives the infinitesimal deformations of the
framed principal $G$--bundle $(E_G,\, \phi)$ obtained by deforming the framing while keeping the
holomorphic principal $G$--bundle $E_G$ fixed.

Now we consider the general case of framings. The subgroups $H_x\, \subsetneq\, G$, $x\, \in
\, D$, in \eqref{hx} are no longer assumed to be trivial.

Consider the subspaces ${\mathcal H}_x\, \subset\,\text{ad}(E_G)_x$, $x\, \in\, D$, constructed
in \eqref{e7}. Let $\text{ad}_\phi(E_G)$ be the holomorphic vector bundle on $X$ defined by the
following short exact sequence of coherent analytic sheaves:
\begin{equation}\label{e9a}
0\, \longrightarrow\, \text{ad}_\phi(E_G)\, \longrightarrow\, \text{ad}(E_G) \, \longrightarrow\,
\bigoplus_{x\in D} \text{ad}(E_G)_x/{\mathcal H}_x\, \longrightarrow\, 0\, ,
\end{equation}
where $\text{ad}(E_G)_x/{\mathcal H}_x$ is supported at $x$. Let
\begin{equation}\label{e9}
\longrightarrow\, H^0(X,\, \text{ad}(E_G)) \, \longrightarrow\,
H^0(X,\, \bigoplus_{x\in D} \text{ad}(E_G)_x/{\mathcal H}_x) \, \stackrel{\widehat{\alpha}_1}{\longrightarrow}\,
H^1(X,\, \text{ad}_\phi(E_G))
\end{equation}
$$
\stackrel{\widehat{\alpha}_2}{\longrightarrow}\, H^1(X,\, \text{ad}(E_G)) \, \longrightarrow\,
H^1(X,\, \bigoplus_{x\in D} \text{ad}(E_G)_x/{\mathcal H}_x)\,=\, 0
$$
be the long exact sequence of cohomologies associated to the short exact sequence of coherent
analytic sheaves in \eqref{e9a}; we have $H^1(X,\, \bigoplus_{x\in D} \text{ad}(E_G)_x/{\mathcal 
H}_x)\,=\, 0$ in \eqref{e9} because $\text{ad}(E_G)_x/{\mathcal H}_x$ is a torsion sheaf
supported on points.

\begin{lemma}\label{lem-3}
\mbox{}
\begin{enumerate}
\item The infinitesimal deformations of any framed holomorphic principal $G$--bundle
$(E_G,\, \phi)$ on $X$ are parametrized by $H^1(X,\, {\rm ad}_\phi(E_G))$, where ${\rm ad}_\phi(E_G)$
is constructed in \eqref{e9a}.

\item The homomorphism $\widehat{\alpha}_2$ in \eqref{e9}
sends an infinitesimal deformation of $(E_G,\, \phi)$ to the infinitesimal deformation
of $E_G$ obtained from it by simply forgetting the framing.

\item Consider the space of framings
${\mathcal F}(E_G)$ on $E_G$ in \eqref{e3}. The tangent space of it at $\phi$ is
$$
T_{\phi}{\mathcal F}(E_G)\,=\, \bigoplus_{x\in D} {\rm ad}(E_G)_x/{\mathcal H}_x
$$
(see \eqref{tp}).
The homomorphism $\widehat{\alpha}_1$ in \eqref{e9} gives all the infinitesimal deformations of the
framed principal $G$--bundle $(E_G,\, \phi)$ obtained by deforming the framing while keeping the
holomorphic principal $G$--bundle $E_G$ fixed.
\end{enumerate}
\end{lemma}

\begin{proof}
First note that for any open subset $U\, \subset\, X$, the space of all holomorphic
sections of $\text{ad}(E_G)\vert_U$ is the space of all holomorphic
vector fields $v$ on $p^{-1}(U)\, \subset\, E_G$, where $p$ is the
projection in \eqref{e2}, satisfying the following
two conditions:
\begin{itemize}
\item $v$ is invariant under the action of $G$ on $E_G$, and

\item $v$ is vertical for the projection $p$.
\end{itemize}
The subsheaf ${\rm ad}_\phi(E_G)\, \subset\, \text{ad}(E_G)$ coincides with the subsheaf that also
preserves the framing $\phi$. The lemma follows from this; we omit the details.
\end{proof}

For any two subgroups $H'\, \subset\, H\, \subsetneq\, G$, and any
holomorphic principal $G$--bundle $F_G$ on $X$,
there is a natural projection $(F_G)_y/H' \, \longrightarrow\, (F_G)_y/H$
for any point $y\,\in\, X$. Therefore, if we have $H'_x\, \subset\, H_x\, \subsetneq\, G$ 
for every $x\, \in\, D$, then a framing of $E_G$ for $\{H'_x\}_{x\in D}$ produces a framing of 
$E_G$ for $\{H_x\}_{x\in D}$. In particular, a framing of $E_G$ for the trivial groups 
$\{e\}_{x\in D}$ produces a framing of $E_G$ for $\{H_x\}_{x\in D}$.

From \eqref{e9a} we conclude that $\text{ad}_\phi(E_G)$ fits in the
following short exact sequence of sheaves on $X$
\begin{equation}\label{g1}
0\, \longrightarrow\, \text{ad}(E_G)(-D)\,:=\, \text{ad}(E_G)\otimes {\mathcal O}_X(-D)
\, \stackrel{\zeta}{\longrightarrow}\, \text{ad}_\phi(E_G) \, \longrightarrow\,
\bigoplus_{x\in D} {\mathcal H}_x\, \longrightarrow\, 0\, .
\end{equation}
Let
$$
\zeta_*\, :\, H^1(X,\, \text{ad}(E_G)(-D))\, \longrightarrow\, H^1(X,\, \text{ad}_\phi(E_G))
$$
be the homomorphism of cohomologies induced by the homomorphism
$\zeta$ in \eqref{g1}. This homomorphism
$\zeta_*$ coincides with the homomorphism of infinitesimal deformations
corresponding to the above map from the framings of a holomorphic principal $G$--bundle
$F_G$ for $\{e\}_{x\in D}$ to the framings of $F_G$ for $\{H_x\}_{x\in D}$.

\subsection{Infinitesimal deformations of a framed $G$--Higgs bundle}

Take a
holomorphic principal $G$--bundle $E_G$ on $X$, and also take a holomorphic section 
$$\theta\, \in\, H^0(X,\, \text{ad}(E_G)\otimes K_X(D))\, .$$ Let
$$
f_\theta\, :\, \text{ad}(E_G)\, \longrightarrow\, \text{ad}(E_G)\otimes K_X(D)
$$
be the ${\mathcal O}_X$--linear homomorphism defined by $t\,\longmapsto\, [\theta,\, t]$.
Now we have the $2$-term complex
\begin{equation}\label{e10}
{\mathcal C}'_{\bullet} \,:\ \
{\mathcal C}'_{0}\,=\, \text{ad}(E_G) \, \stackrel{f_\theta}{\longrightarrow}\,
{\mathcal C}'_{1}\,=\, \text{ad}(E_G)\otimes K_X(D)\, ,
\end{equation}
where ${\mathcal C}'_{i}$ is at the $i$-th position.

The following lemma is proved in \cite[p.~220, Theorem 2.3]{BR}, \cite[p.~399, Proposition 3.1.2]{Bot},
\cite[p.~271, Proposition 7.1]{Ma} (see also \cite{Bi}).

\begin{lemma}\label{lem-1}
The infinitesimal deformations of the $D$--twisted $G$--Higgs bundle
$(E_G,\, \theta)$ are parametrized by elements
of the first hypercohomology ${\mathbb H}^1({\mathcal C}'_{\bullet})$, where ${\mathcal C}'_{\bullet}$
is the complex in \eqref{e10}.
\end{lemma}

The following lemma gives the dimension of the infinitesimal deformations.

\begin{lemma}\label{lemn1}
Assume that ${\rm genus}(X)\, \geq\,1$. Let $(E_G,\, \theta)$ be a stable
$D$--twisted $G$--Higgs bundle. Then
$$
{\mathbb H}^0({\mathcal C}'_{\bullet})\,=\, \{v\, \in\, H^0(X,\, {\rm ad}(E_G))\, \mid\,
[\theta,\, v]\,=\, 0\}\,=\, Z({\mathfrak g})\, ,
$$
where $Z({\mathfrak g})\, \subset\, \mathfrak g$ as before is the center, and
$$
{\mathbb H}^2({\mathcal C}'_{\bullet})\,=\, 0\, .
$$
Moreover, $$\dim {\mathbb H}^1({\mathcal C}'_{\bullet})\,=\,\dim G\cdot (2\cdot({\rm genus}(X)-1)
+n)+\dim Z({\mathfrak g})\, ,$$ where $n\, =\,\# D$.
\end{lemma}

\begin{proof}
Consider the short exact sequence of complexes of sheaves
$$
\begin{matrix}
& & 0 & & 0\\
& & \Big\downarrow & & \Big\downarrow \\ 
& & 0 & {\longrightarrow} & \text{ad}(E_G)\otimes K_X(D)\\
&& \Big\downarrow && \,\,\,\,\,\,\,\Big\downarrow=\\
{\mathcal C}'_{\bullet} & : & \text{ad}(E_G) & \stackrel{f_\theta}{\longrightarrow} &
\text{ad}(E_G)\otimes K_X(D)\\
&& \,\,\,\,\,\,\,\Big\downarrow= && \Big\downarrow\\
&& \text{ad}(E_G) & \longrightarrow & 0\\
& & \Big\downarrow & & \Big\downarrow \\
& & 0 & & 0
\end{matrix}
$$
on $X$. Let
\begin{equation}\label{les1}
0\, \longrightarrow\, {\mathbb H}^0({\mathcal C}'_{\bullet}) \, \longrightarrow\, H^0(X,\, \text{ad}(E_G))
\, \longrightarrow\, H^0(X,\, \text{ad}(E_G)\otimes K_X(D)) \, \longrightarrow\,
{\mathbb H}^1({\mathcal C}'_{\bullet}) 
\end{equation}
$$
\, \longrightarrow\, H^1(X,\, \text{ad}(E_G)) \, \stackrel{\varpi}{\longrightarrow}\, H^1(X,\,
\text{ad}(E_G)\otimes K_X(D))
\, \longrightarrow\, {\mathbb H}^2({\mathcal C}'_{\bullet})\, \longrightarrow\,0
$$
be the long exact sequence of hypercohomologies associated it. First note that
the trivial holomorphic vector bundle $X\times Z({\mathfrak g})$ over $X$ is a holomorphic
subbundle of $\text{ad}(E_G)$, because the adjoint action of $G$ on $\mathfrak g$ fixes
$Z({\mathfrak g})$ pointwise. The stability condition of $(E_G,\, \theta)$ implies that
\begin{equation}\label{sc}
\{v\, \in\, H^0(X,\, {\rm ad}(E_G))\, \mid\,
[\theta,\, v]\,=\, 0\}\,=\, H^0(X,\, X\times Z({\mathfrak g}))\,=\, Z({\mathfrak g})\, .
\end{equation}
On the other hand, from \eqref{les1} it follows that $${\mathbb H}^0({\mathcal C}'_{\bullet})
\,=\, \{v\, \in\, H^0(X,\, {\rm ad}(E_G))\, \mid\, [\theta,\, v]\,=\, 0\}\, .$$ Hence, we have
that ${\mathbb H}^0({\mathcal C}'_{\bullet})\,=\, Z({\mathfrak g})$.

Next note that from \eqref{sc} it follows that
\begin{equation}\label{sc2}
\{v\, \in\, H^0(X,\, {\rm ad}(E_G)\otimes{\mathcal O}_X(-D))\, \mid\,
[\theta,\, v]\,=\, 0\}\,=\, 0\, .
\end{equation}
The nondegenerate symmetric bilinear form $\widehat{\sigma}$ in \eqref{e6} identifies the holomorphic
vector bundle $\text{ad}(E_G)$ with its dual $\text{ad}(E_G)^*$. Hence
Serre duality gives that $$H^1(X,\, \text{ad}(E_G))\,=\, H^0(X,\, \text{ad}(E_G)\otimes K_X)^*$$
and $H^1(X,\, \text{ad}(E_G)\otimes K_X(D))\,=\, H^0(X,\,\text{ad}(E_G)\otimes
{\mathcal O}_X(-D))^*$. Using these isomorphisms, the homomorphism $\varpi$ in \eqref{les1}
coincides with the dual of the homomorphism
\begin{equation}\label{ws}
H^0(X,\,\text{ad}(E_G)\otimes
{\mathcal O}_X(-D))\, \longrightarrow\, H^0(X,\, \text{ad}(E_G)\otimes K_X),\, \ \
v\, \longmapsto\, [\theta,\, v]\, .
\end{equation}
Therefore, the homomorphism in \eqref{ws} will be denoted by $\varpi^*$.
{}From \eqref{sc2} it now follows that $\varpi^*$ in \eqref{ws}
is injective. Hence its dual $\varpi$ is surjective.
Consequently, from \eqref{les1} we now conclude that
${\mathbb H}^2({\mathcal C}'_{\bullet})\,=\, 0$.

{}From \eqref{les1} it follows immediately that
$$
\dim {\mathbb H}^1({\mathcal C}'_{\bullet})\,=\, \chi(\text{ad}(E_G)\otimes K_X(D))
-\chi(\text{ad}(E_G))+\dim {\mathbb H}^0({\mathcal C}'_{\bullet})
+\dim {\mathbb H}^2({\mathcal C}'_{\bullet})\, ,
$$
where $\chi(F)\,:=\, \dim H^0(X,\, F)- \dim H^1(X,\, F)$ is the Euler characteristic.
By Riemann--Roch, we have
$\chi(\text{ad}(E_G))\,=\,\dim G\cdot (1-\text{genus}(X))$, and
$$
\chi(\text{ad}(E_G)\otimes K_X(D))\,=\, -\chi(\text{ad}(E_G)\otimes{\mathcal O}_X(-D))
\,=\, \dim G\cdot (\text{genus}(X)-1+n)\, .
$$
Consequently, from the above
computations of ${\mathbb H}^0({\mathcal C}'_{\bullet})$ and ${\mathbb H}^2({\mathcal C}'_{\bullet})$
it follows that $\dim {\mathbb H}^1({\mathcal C}'_{\bullet})
\,=\,\dim G\cdot (2\cdot ({\rm genus}(X)-1) +n)+\dim Z({\mathfrak g})$.
\end{proof}

\begin{remark}\label{remn1}
For a stable $D$--twisted $G$--Higgs bundle
$(E_G,\, \theta)$, since ${\mathbb H}^2({\mathcal C}'_{\bullet})\,=\, 0$, the
deformations of $(E_G,\, \theta)$ are unobstructed.
\end{remark}

We shall describe the space of all infinitesimal deformations of a framed $G$--Higgs bundle.
For that, we first consider the special case where $H_x\, =\, \{e\}$ for every $x\, \in\, D$.

Consider the following $2$-term sub-complex of the complex ${\mathcal C}'_{\bullet}$ in \eqref{e10}:
\begin{equation}\label{e11}
{\mathcal C}_{\bullet} \,:\ \
{\mathcal C}_{0}\,=\, \text{ad}(E_G)(-D)\, :=\, \text{ad}(E_G)\otimes {\mathcal O}_X(-D)
\, \stackrel{f_\theta}{\longrightarrow}\,{\mathcal C}_{1}\,=\, \text{ad}(E_G)\otimes K_X(D)
\end{equation}
(here the restriction of the homomorphism $f_\theta$ to $\text{ad}(E_G)(-D) \,\subset\,
\text{ad}(E_G)$ is also denoted by $f_\theta$). Let $\phi$ be a framing on $E_G$. Since $H_x\,=\,
e$ for all $x\, \in\, D$, we have that ${\mathcal H}_x\,=\, 0$, which implies that
$\text{ad}^n_\phi(E_G)\,=\, \text{ad}(E_G)$. Consequently, the triple $(E_G,\, \phi,\,
\theta)$ is a framed $G$--Higgs bundle.

\begin{lemma}\label{lemid}
Assume that $H_x\, =\, \{e\}$ for every $x\, \in\, D$.
The infinitesimal deformations of the framed $G$--Higgs bundle $(E_G,\, \phi,\, \theta)$ are
parametrized by elements
of the first hypercohomology ${\mathbb H}^1({\mathcal C}_{\bullet})$, where ${\mathcal C}_{\bullet}$
is the complex in \eqref{e11}.
\end{lemma}

\begin{proof}
The proof of Lemma \ref{lem-1} also works for this lemma after some very minor
and straight-forward modifications. (In the special case where $G\,=\, \text{GL}(r,{\mathbb C})$,
this lemma reduces to Lemma 2.7 of \cite{BLP}.)
\end{proof}

We have the following short exact sequence of complexes of sheaves on $X$
$$
\begin{matrix}
0 && 0 && 0\\
\Big\downarrow && \Big\downarrow && \Big\downarrow\\
{\mathcal C}_{\bullet} & : & \text{ad}(E_G)(-D) & \stackrel{f_\theta}{\longrightarrow} &
\text{ad}(E_G)\otimes K_X(D)\\
\Big\downarrow && \Big\downarrow && \,\,\,\,\,\,\,\Big\downarrow =\\
{\mathcal C}'_{\bullet} & : & \text{ad}(E_G) & \stackrel{f_\theta}{\longrightarrow} &
\text{ad}(E_G)\otimes K_X(D)\\
\Big\downarrow && \Big\downarrow && \Big\downarrow\\
{\mathcal C}''_{\bullet} &: & {\mathcal C}''_0\,=\, \text{ad}(E_G)\vert_D & \longrightarrow & 0\\
\Big\downarrow && \Big\downarrow && \Big\downarrow\\
0 && 0 && 0
\end{matrix}
$$
where ${\mathcal C}''_{\bullet}$ is a $1$-term complex, and
${\mathcal C}'_{\bullet}$ is defined in \eqref{e10}. Let
\begin{equation}\label{ahb}
\longrightarrow\,
{\mathbb H}^0(X,\, {\mathcal C}''_{\bullet})\,=\, H^0(X,\, \text{ad}(E_G)\vert_D)\,
\stackrel{\beta'_1}{\longrightarrow}\,
{\mathbb H}^1(X,\, {\mathcal C}_{\bullet})\,
\stackrel{\beta'_2}{\longrightarrow}\, {\mathbb H}^1(X,\, {\mathcal C}'_{\bullet}) 
\, \longrightarrow\, {\mathbb H}^1(X,\, {\mathcal C}''_{\bullet})\,=\, 0
\end{equation}
be the long exact sequence of hypercohomologies associated to this short exact sequence of complexes;
note that we have ${\mathbb H}^1(X,\, {\mathcal C}''_{\bullet})\,=\, 0$ because ${\mathcal C}''_0$ is a
torsion sheaf supported on points. The homomorphism
$\beta'_2$ in \eqref{ahb} coincides with the homomorphism of infinitesimal deformations 
corresponding to the forgetful map that sends any
framed $G$--Higgs bundle $(E'_G,\, \phi',\, \theta')$ to the
$D$--twisted $G$--Higgs bundle $(E'_G,\, \theta')$ by
forgetting the framing; see Lemma \ref{lem-1} and Lemma \ref{lemid}
(we have $H_x\, =\, \{e\}$ for every $x\, \in\, D$). The homomorphism
$\beta'_1$ in \eqref{ahb} corresponds to moving just the framing while keeping the
$D$--twisted $G$--Higgs bundle $(E_G,\, \theta)$ fixed.

Now we consider framings of general type. The subgroups $H_x\, \subsetneq\, G$, $x\, \in
\, D$, in \eqref{hx} are no longer assumed to be trivial.

Let $(E_G,\, \phi,\, \theta)$ be a framed $G$--Higgs bundle. Consider the subspace
${\mathcal H}^\perp_x\, \subset\,\text{ad}(E_G)_x$ in \eqref{f1}. Let $\text{ad}^n_\phi(E_G)$ be
the holomorphic vector bundle on $X$ defined by the following short exact sequence of coherent
analytic sheaves on $X$:
\begin{equation}\label{e12}
0\, \longrightarrow\, \text{ad}^n_\phi(E_G)\, \longrightarrow\, \text{ad}(E_G) \, \longrightarrow\,
\bigoplus_{x\in D} \text{ad}(E_G)_x/{\mathcal H}^\perp_x\, \longrightarrow\, 0\, ,
\end{equation}
where $\text{ad}(E_G)_x/{\mathcal H}^\perp_x$ is supported at $x$. From
\eqref{e12} it follows immediately that the holomorphic sections of $\text{ad}^n_\phi(E_G)$ are precisely
the sections $s\, \in\, H^0(X,\, \text{ad}(E_G))$ such that $s(x)\, \in\, {\mathcal H}^\perp_x$ for every $x\, \in\, D$.
Hence from the definition of Higgs fields on $(E_G,\, \phi)$ it follows that
Higgs fields on $(E_G,\, \phi)$ are
precisely the holomorphic sections of the holomorphic vector
bundle $\text{ad}^n_\phi(E_G)\otimes K_X(D)$.

\begin{lemma}\label{lem1}
The homomorphism $f_\theta$ in \eqref{e10} sends the subsheaf ${\rm ad}_\phi(E_G)\, \subset\,
{\rm ad}(E_G)$ constructed in \eqref{e9a} to the subsheaf ${\rm ad}^n_\phi(E_G)\otimes K_X(D)\, \subset\,
{\rm ad}(E_G)\otimes K_X(D)$, where ${\rm ad}^n_\phi(E_G)$ is constructed in \eqref{e12}.
\end{lemma}

\begin{proof}
Let $S$ be a Lie subalgebra of $\mathfrak g$. Let $S^\perp$ be the annihilator of it for the symmetric
bilinear form $\sigma$ in \eqref{e5}. Then it can be shown that
\begin{equation}\label{li}
[S,\, S^\perp]\, \subset\, S^\perp\, .
\end{equation}
Indeed, the $G$--invariance condition on $\sigma$ implies that
\begin{equation}\label{gi}
\sigma([a,\, c]\otimes b)\,+\, \sigma(c\otimes [a,\, b])\,=\, 0
\end{equation}
for all $a,\, b,\, c\, \in\, \mathfrak g$. In particular, for any $a,\, b\, \in\, S$ and $c\, \in\, S^\perp$,
$$
\sigma([a,\, c]\otimes b) \,=\,-\, \sigma(c\otimes [a,\, b])\,=\, 0\, ,
$$
because $[b,\, a]\, \in\, S$ and $c\, \in\, S^\perp$.

For any $x\, \in\, D$, the image of the homomorphism ${\rm ad}_\phi(E_G)_x\, \longrightarrow\,
{\rm ad}(E_G)_x$ in \eqref{e9a} is ${\mathcal H}_x$, while the image of the homomorphism
${\rm ad}^n_\phi(E_G)_x\, \longrightarrow\,
{\rm ad}(E_G)_x$ in \eqref{e12} is ${\mathcal H}^\perp_x$. From \eqref{li} we know
that
\begin{equation}\label{e122}
[{\mathcal H}_x,\, {\mathcal H}^\perp_x]\, \subset\, {\mathcal H}^\perp_x\, .
\end{equation}
Since $\theta$ is a holomorphic section of $\text{ad}^n_\phi(E_G)\otimes K_X(D)$,
the lemma follows from \eqref{e122}.
\end{proof}

The restriction of $f_\theta$ (defined in \eqref{e10}) to ${\rm ad}_\phi(E_G)\, \subset\,
{\rm ad}(E_G)$ will be denoted by $f^0_\theta$. Let ${\mathcal D}_{\bullet}$ be
the following $2$-term sub-complex of ${\mathcal C}'_{\bullet}$ constructed in \eqref{e10}:
\begin{equation}\label{e13}
{\mathcal D}_{\bullet} \,:\ \
{\mathcal D}_{0}\,=\, \text{ad}_\phi(E_G) \, \stackrel{f^0_\theta}{\longrightarrow}\,
{\mathcal D}_{1}\,=\, \text{ad}^n_\phi(E_G)\otimes K_X(D)
\end{equation}
(Lemma \ref{lem1} shows that $f_\theta({\mathcal D}_{0})\, \subset\, {\mathcal D}_{1}$).

\begin{lemma}\label{lem-2}
All the infinitesimal deformations of the given framed $G$--Higgs bundle $(E_G,\, \phi,\, \theta)$ are
parametrized by the elements of the first hypercohomology ${\mathbb H}^1({\mathcal D}_{\bullet})$, where
${\mathcal D}_{\bullet}$ is constructed in \eqref{e13}.
\end{lemma}

\begin{proof}
Just as the proof of Lemma \ref{lem-1} also works for Lemma \ref{lemid}, it works even for this lemma
after the framing is suitably taken into account. We omit the details. It should be mentioned that this
lemma can also be proved using the framework of Section \ref{se6}.
\end{proof}

We have the following short exact sequence of complexes of sheaves on $X$
\begin{equation}\label{des}
\begin{matrix}
0 && 0 && 0\\
\Big\downarrow && \Big\downarrow && \Big\downarrow\\
{\mathcal D}'_{\bullet} & : & 0 & \longrightarrow & {\mathcal D}'_{1}\,=\,
\text{ad}^n_\phi(E_G)\otimes K_X(D)\\
\Big\downarrow &&\Big\downarrow && \,\,\,\,\,\,\,\Big\downarrow =\\
{\mathcal D}_{\bullet} & : & \text{ad}_\phi(E_G) & \stackrel{f^0_\theta}{\longrightarrow} &
\text{ad}^n_\phi(E_G)\otimes K_X(D)\\
\Big\downarrow && \,\,\,\,\,\,\, \Big\downarrow = && \Big\downarrow\\
{\mathcal D}''_{\bullet} & : & {\mathcal D}''_0\,=\, \text{ad}_\phi(E_G) & \longrightarrow & 0\\
\Big\downarrow && \Big\downarrow && \Big\downarrow\\
0 && 0 && 0
\end{matrix}
\end{equation}
(both ${\mathcal D}'_{\bullet}$ and ${\mathcal D}''_{\bullet}$ are $1$-term complexes concentrated
at the first position and the $0$-th position respectively). Let
\begin{equation}\label{g3}
\longrightarrow\, {\mathbb H}^0({\mathcal D}''_{\bullet})\,=\, H^0(X,\, \text{ad}_\phi(E_G))
\, \longrightarrow\,
{\mathbb H}^1({\mathcal D}'_{\bullet})\,=\, H^0(X,\, \text{ad}^n_\phi(E_G)\otimes K_X(D))\,
\longrightarrow
\end{equation}
$$
\stackrel{\beta'_3}{\longrightarrow}
\, {\mathbb H}^1({\mathcal D}_{\bullet}) \, \stackrel{\beta'_4}{\longrightarrow}\,
{\mathbb H}^1({\mathcal D}''_{\bullet})\,=\, H^1(X,\, \text{ad}_\phi(E_G))
$$
be the long exact sequence of hypercohomologies associated to
\eqref{des}. The homomorphism $\beta'_3$ in \eqref{g3} corresponds to deforming the Higgs field 
keeping the framed principal bundle $(E_G,\, \phi)$ fixed; recall that the Higgs fields on
$(E_G,\, \phi)$ are the holomorphic sections of $\text{ad}^n_\phi(E_G)\otimes K_X(D)$.
The homomorphism $\beta'_4$ in \eqref{g3} corresponds
to the forgetful map that sends an infinitesimal deformation of $(E_G,\, \phi,\, \theta)$
to the infinitesimal deformation of $(E_G,\, \phi)$ it gives by simply forgetting the
Higgs field (see Lemma \ref{lem-3}(1)).

The hypercohomologies of ${\mathcal D}_{\bullet}$ will be computed in
Section \ref{se5.1}.

\section{Framed $G$--Higgs bundles and symplectic geometry}

\subsection{Construction of a symplectic structure}

\begin{proposition}\label{prop1}
The dual ${\rm ad}_\phi(E_G)^*$ of the vector bundle ${\rm ad}_\phi(E_G)$ in \eqref{e9a}
is identified with ${\rm ad}^n_\phi(E_G)\otimes {\mathcal O}_X(D)$, where
${\rm ad}^n_\phi(E_G)$ is constructed in \eqref{e12}. This identification is canonical
in the sense that it depends only on $\sigma$ in \eqref{e5}.

The dual vector bundle ${\rm ad}^n_\phi(E_G)^*$ is identified with
${\rm ad}_\phi(E_G)\otimes {\mathcal O}_X(D)$; this identification is canonical in the
above sense.
\end{proposition}

\begin{proof}
Consider the fiberwise nodegenerate symmetric bilinear form $$\widehat{\sigma}\, :\,
{\rm ad}(E_G)^{\otimes 2}\, \longrightarrow\,{\mathcal O}_X$$ in \eqref{e6}. Tensoring
it with ${\mathcal O}_X(D)$ We get the homomorphism
\begin{equation}\label{eqzh}
\widehat{\sigma}\otimes{\rm Id}_{{\mathcal O}_X(D)}\, :\,
{\rm ad}(E_G)\otimes {\rm ad}(E_G)\otimes {\mathcal O}_X(D)\, \longrightarrow\,{\mathcal O}_X(D)\, .
\end{equation}
Recall
that both ${\rm ad}_\phi(E_G)$ and ${\rm ad}^n_\phi (E_G)$ are contained in
${\rm ad}(E_G)$ (see \eqref{e9a} and \eqref{e12}). It can be shown that
the image of the restriction of the homomorphism $\widehat{\sigma}\otimes{\rm Id}_{{\mathcal O}_X(D)}$
in \eqref{eqzh} to the subsheaf
$$
{\rm ad}_\phi(E_G)\otimes {\rm ad}^n_\phi (E_G)\otimes {\mathcal O}_X(D)
\,\subset\, {\rm ad}(E_G)\otimes {\rm ad}(E_G)\otimes {\mathcal O}_X(D)
$$
is contained in ${\mathcal O}_X\, \subset\, {\mathcal O}_X(D)$. Indeed, this follows from the
facts that for any $x\, \in\, D$, the image of ${\rm ad}_\phi(E_G)_x$ (respectively,
${\rm ad}^n_\phi (E_G)_x$) in ${\rm ad}(E_G)_x$ is ${\mathcal H}_x$ (respectively, ${\mathcal H}^\perp_x$),
and ${\mathcal H}_x$ annihilates ${\mathcal H}^\perp_x$ for the form $\widehat{\sigma}(x)$.

The above restricted homomorphism
$$
\widehat{\sigma}\otimes{\rm Id}_{{\mathcal O}_X(D)}\, :\,
{\rm ad}_\phi(E_G)\otimes {\rm ad}^n_\phi (E_G)\otimes {\mathcal O}_X(D)
\,\longrightarrow\, {\mathcal O}_X
$$
produces a homomorphism
\begin{equation}\label{S}
S\, :\, {\rm ad}^n_\phi (E_G)\otimes {\mathcal O}_X(D)\, \longrightarrow\,
{\rm ad}_\phi(E_G)^*\, .
\end{equation}
This homomorphism $S$ is an isomorphism over the complement $X\setminus D$, because 
\begin{itemize}
\item the pairing $\widehat{\sigma}$ in \eqref{e6} is fiberwise nondegenerate, and

\item ${\rm ad}^n_\phi (E_G)\vert_{X\setminus D}\,=\, \text{ad}(E_G)\vert_{X\setminus D} \,=\,
{\rm ad}_\phi(E_G)\vert_{X\setminus D}$.
\end{itemize}
Denote the torsion sheaf ${\rm ad}_\phi(E_G)^*/{\rm Image}(S)$ by $Q$, where $S$
is the homomorphism in \eqref{S}. So we have
\begin{equation}\label{deg}
\text{degree}(Q)\,=\, \text{degree}({\rm ad}_\phi(E_G)^*)- \text{degree}
({\rm ad}^n_\phi (E_G)\otimes {\mathcal O}_X(D))\, .
\end{equation}

Since $\widehat{\sigma}$ in \eqref{e6} produces an isomorphism of $\text{ad}(E_G)$ with
the dual vector bundle $\text{ad}(E_G)^*$, it follows that $\text{degree}(\text{ad}(E_G))
\,=\, 0$. Hence from \eqref{e9a} it follows immediately that
$$
\text{degree}({\rm ad}_\phi(E_G))\,=\, \sum_{x\in D} (\dim {\mathcal H}_x-\dim {\mathfrak g})\, , 
$$
while from \eqref{e12} it follows that
$$
\text{degree}({\rm ad}^n_\phi(E_G))\,=\, \sum_{x\in D} (\dim {\mathcal H}^\perp_x-
\dim {\mathfrak g})\, .
$$
Consequently, we have $\text{degree}({\rm ad}^n_\phi(E_G)\otimes{\mathcal O}_X(D))\,=\,
\sum_{x\in D} \dim {\mathcal H}^\perp_x$.
As $\dim {\mathcal H}_x+ \dim {\mathcal H}^\perp_x\,=\, \dim {\mathfrak g}$, from \eqref{deg}
it now follows that $\text{degree}(Q)\,=\, 0$. Since $Q$ is a torsion sheaf
with $\text{degree}(Q)\,=\, 0$, we conclude that $Q\,=\, 0$. Consequently, the
homomorphism $S$ in \eqref{S} is an isomorphism. This proves the first statement of the
proposition.

The isomorphism in the second statement of the proposition is
given by $S^* \otimes{\rm Id}_{{\mathcal O}_X(D)}$.
\end{proof}

\begin{remark}\label{remn2}
Consider the dual of the homomorphism ${\rm ad}_\phi(E_G)\, \hookrightarrow\, {\rm ad}(E_G)$
in \eqref{e9a}. From the first statement in Proposition \ref{prop1} we have
$$
{\rm ad}(E_G)\,=\,
{\rm ad}(E_G)^*\, \hookrightarrow\, {\rm ad}_\phi(E_G)^*\, \longrightarrow\,
{\rm ad}^n_\phi(E_G)\otimes {\mathcal O}_X(D)\, ;
$$
as in the proof of Proposition \ref{prop1}, the two vector bundles ${\rm ad}(E_G)$ and ${\rm ad}(E_G)^*$
are identified using $\widehat{\sigma}$.
\end{remark}

Consider the complex ${\mathcal D}_{\bullet}$ in \eqref{e13}. Its Serre dual complex, which
we shall denote by ${\mathcal D}^\vee_{\bullet}$, is the following:
\begin{equation}\label{e14}
{\mathcal D}^\vee_{\bullet} \,:\ \
{\mathcal D}^\vee_{0}\,=\, (\text{ad}^n_\phi(E_G)\otimes K_X(D))^*\otimes K_X
\, \stackrel{(f^0_\theta)^*\otimes{\rm Id}_{K_X}}{\longrightarrow}\,
{\mathcal D}^\vee_{1}\,=\, \text{ad}_\phi(E_G)^*\otimes K_X\, ;
\end{equation}
to clarify, ${\mathcal D}^\vee_{0}$ and ${\mathcal D}^\vee_{1}$ are at the $0$-th position and $1$-st
position respectively. From Proposition \ref{prop1} it follows that
\begin{itemize}
\item ${\mathcal D}^\vee_{0}\,=\,\text{ad}_\phi(E_G)$, and

\item ${\mathcal D}^\vee_1\,=\,\text{ad}^n_\phi(E_G)\otimes K_X(D)$.
\end{itemize}
Moreover, using these two identifications, the homomorphism $(f^0_\theta)^*\otimes{\rm Id}_{K_X}$ in \eqref{e14}
coincides with $f^0_\theta$. In other words, the dual complex
${\mathcal D}^\vee_{\bullet}$ is canonically identified with ${\mathcal D}_{\bullet}$;
this isomorphism of course depends on the bilinear form $\sigma$ in \eqref{e5}. Let
\begin{equation}\label{xi}
\xi\, :\, {\mathcal D}_{\bullet}\, \stackrel{\sim}{\longrightarrow}\,
{\mathcal D}^\vee_{\bullet}
\end{equation}
be this isomorphism. This isomorphism $\xi$ produces
an isomorphism
\begin{equation}\label{wtP}
\widetilde{\Phi}\,:=\, \xi_*\, :\, {\mathbb H}^1({\mathcal D}_{\bullet})\,
\stackrel{\sim}{\longrightarrow}\, {\mathbb H}^1({\mathcal D}^\vee _{\bullet})
\end{equation}
of hypercohomologies. On the other hand, Serre duality gives that
$$
{\mathbb H}^1({\mathcal D}^\vee _{\bullet})\,=\, {\mathbb H}^1({\mathcal D}_{\bullet})^*
$$
(cf. \cite[p.~67, Theorem 3.12]{Huy}).
Using this, the isomorphism $\widetilde{\Phi}$ in \eqref{wtP} produces an isomorphism
\begin{equation}\label{e15}
\Phi_{(E_G, \phi, \theta)} \, :\, {\mathbb H}^1({\mathcal D}_{\bullet})\,
\stackrel{\sim}{\longrightarrow}\, {\mathbb H}^1({\mathcal D} _{\bullet})^*\, .
\end{equation}
This homomorphism $\Phi_{(E_G, \phi, \theta)}$ is clearly skew-symmetric.

We shall now describe an alternative construction of the
homomorphism $\Phi_{(E_G, \phi, \theta)}$ in \eqref{e15}.

Consider the tensor product of complexes $\widehat{\mathcal D}_{\bullet}\, :=\,
{\mathcal D}_{\bullet}\otimes {\mathcal D}_{\bullet}$. So
$$
\widehat{\mathcal D}_{\bullet}\, :\, \widehat{\mathcal D}_0
\,=\, {\rm ad}_\phi(E_G)\otimes {\rm ad}_\phi(E_G)\,
\stackrel{f^0_\theta\otimes{\rm Id}+ {\rm Id}\otimes f^0_\theta}{\longrightarrow}\,
\widehat{\mathcal D}_1
$$
$$
=\, ((\text{ad}^n_\phi(E_G)\otimes K_X(D))\otimes {\rm ad}_\phi(E_G))\oplus ({\rm ad}_\phi(E_G)\otimes
(\text{ad}^n_\phi(E_G)\otimes K_X(D)))
$$
$$
\stackrel{{\rm Id}\otimes f^0_\theta- f^0_\theta\otimes {\rm Id}}{\longrightarrow}\,
\widehat{\mathcal D}_2\,=\,
(\text{ad}^n_\phi(E_G)\otimes K_X(D))\otimes (\text{ad}^n_\phi(E_G)\otimes K_X(D))\, ;
$$
to clarify, $\widehat{\mathcal D}_i$ is at the $i$-th position. Consider the homomorphism
$$
\gamma\, :\, \widehat{\mathcal D}_1 \,=\, ((\text{ad}^n_\phi(E_G)\otimes K_X(D))
\otimes {\rm ad}_\phi(E_G))\oplus ({\rm ad}_\phi(E_G)\otimes
(\text{ad}^n_\phi(E_G)\otimes K_X(D)))
$$
$$
\longrightarrow\, K_X\, , \ \ (a\otimes b) +(c\otimes d)\, \longmapsto\,
\widehat{\sigma}(a\otimes b)+ \widehat{\sigma}(c\otimes d)\, ,
$$
where $\widehat{\sigma}$ is the pairing in \eqref{e6}. Using \eqref{gi} it is
straight-forward to deduce that
$$
\gamma\circ (f^0_\theta\otimes{\rm Id}+ {\rm Id}\otimes f^0_\theta)\,=\, 0\, .
$$
Consequently, $\gamma$ produces a homomorphism of complexes
\begin{equation}\label{e16}
\Gamma\, :\, \widehat{\mathcal D}_{\bullet}\, \longrightarrow\,K_X[-1]\, ,
\end{equation}
where $K_X[-1]$ is the complex $0\, \longrightarrow\, K_X$, with $K_X$ being at the
$1$-st position. More precisely, $\Gamma$ is the following homomorphism of complexes:
$$
\begin{matrix}
\widehat{\mathcal D}_{\bullet} & : & 
\widehat{\mathcal D}_0 & \longrightarrow & \widehat{\mathcal D}_1 & \longrightarrow &
\widehat{\mathcal D}_2\\
~\Big\downarrow\Gamma && \Big\downarrow && ~\Big\downarrow\gamma && \Big\downarrow\\
K_X[-1] & : & 
0 & \longrightarrow & K_X & \longrightarrow & 0
\end{matrix}
$$
Now we have the homomorphisms of hypercohomologies
\begin{equation}\label{e17}
{\mathbb H}^1({\mathcal D}_{\bullet})\otimes {\mathbb H}^1({\mathcal D}_{\bullet})\,
\longrightarrow\,{\mathbb H}^2({\mathcal D}_{\bullet}\otimes {\mathcal D}_{\bullet})\,=\,
{\mathbb H}^2(\widehat{\mathcal D}_{\bullet})\, \stackrel{\Gamma_*}{\longrightarrow}\,
{\mathbb H}^2(K_X[-1])
\end{equation}
$$
=\, H^1(X,\, K_X)\,=\, H^1(X,\, {\mathcal O}_X)^*\,=\, {\mathbb C}\, ,
$$
where $\Gamma_*$ is the homomorphism of hypercohomologies induced by the homomorphism $\Gamma$
in \eqref{e16}.

The bilinear form on
${\mathbb H}^1({\mathcal D}_{\bullet})$ constructed in \eqref{e17} coincides with the one given by
the isomorphism $\Phi_{(E_G, \phi, \theta)}$ in \eqref{e15}.

Recall from Lemma \ref{lem-2} that the space of infinitesimal deformations of
$(E_G,\, \phi,\, \theta)$ is identified with ${\mathbb H}^1({\mathcal D}_{\bullet})$.

The above constructions are summarized in the following lemma:

\begin{lemma}\label{lem2}
The space of infinitesimal deformations of any given framed $G$--Higgs bundle $(E_G,\, \phi,\, 
\theta)$, namely ${\mathbb H}^1({\mathcal D}_{\bullet})$, is equipped with a natural symplectic 
structure $\Phi_{(E_G, \phi, \theta)}$ that is constructed in \eqref{e15} (and also in \eqref{e17}).
\end{lemma}

\subsection{A Poisson structure}

Take a $D$--twisted $G$--Higgs bundle $(E_G,\, \theta)$ as in Lemma \ref{lem-1}.
Consider the hypercohomology ${\mathbb H}^1({\mathcal C}'_{\bullet})$, where ${\mathcal C}'_{\bullet}$
is constructed in \eqref{e10}.
Following \cite{Bot}, \cite{BR}, \cite{Ma}, we shall show that there is a natural homomorphism to
it from its dual ${\mathbb H}^1({\mathcal C}'_{\bullet})^*$.

\begin{proposition}\label{prop2}
There is a natural homomorphism
$$
P_{(E_G, \theta)}\, :\, {\mathbb H}^1({\mathcal C}'_{\bullet})^*\, \longrightarrow\,
{\mathbb H}^1({\mathcal C}'_{\bullet})\, .
$$
\end{proposition}

\begin{proof}
Let $({\mathcal C}')^\vee_{\bullet}$ denote the Serre dual complex of ${\mathcal C}'_{\bullet}$
in \eqref{e10}. So we have
\begin{equation}\label{e18}
({\mathcal C}')^\vee_{\bullet} \,:\ \ ({\mathcal C}')^\vee_{0}\,=\, (\text{ad}(E_G)\otimes K_X(D))^*\otimes K_X
\, \stackrel{f^*_\theta\otimes {\rm Id}_{K_X}}{\longrightarrow}\,
({\mathcal C}')^\vee_{1}\,=\, \text{ad}(E_G)^*\otimes K_X\, ,
\end{equation}
where $({\mathcal C}')^\vee_{i}$ is at the $i$-th position. As done before, the form
$\widehat{\sigma}$ in \eqref{e6} identifies $\text{ad}(E_G)^*$ with $\text{ad}(E_G)$. So
$$({\mathcal C}')^\vee_{0}\,=\, \text{ad}(E_G)\otimes {\mathcal O}_X(-D)\ \ \text{ and }\ \
({\mathcal C}')^\vee_1\,=\, \text{ad}(E_G)\otimes K_X\, .$$
Using these two identifications, the homomorphism 
$f^*_\theta\otimes {\rm Id}_{K_X}$ in \eqref{e18} coincides with the restriction of $f_\theta$
to the subsheaf $\text{ad}(E_G)\otimes {\mathcal O}_X(-D)\, \subset\, \text{ad}(E_G)$; this restriction
of $f_\theta$ to $\text{ad}(E_G)\otimes {\mathcal O}_X(-D)$ will also be denoted by $f_\theta$. Hence
the complex $({\mathcal C}')^\vee_{\bullet}$ in \eqref{e18} becomes
\begin{equation}\label{ad}
({\mathcal C}')^\vee_{\bullet} \,:\ \ ({\mathcal C}')^\vee_{0}\,=\, \text{ad}(E_G)\otimes {\mathcal O}_X(-D)
\, \stackrel{f_\theta}{\longrightarrow}\,
({\mathcal C}')^\vee_{1}\,=\, \text{ad}(E_G)\otimes K_X\, .
\end{equation}
Consequently, we have a homomorphism of complexes ${\mathcal R}\, :\, ({\mathcal C}')^\vee_{\bullet}\, \longrightarrow\,
{\mathcal C}'_{\bullet}$ defined by
$$
\begin{matrix}
({\mathcal C}')^\vee_{\bullet} & : & 
\text{ad}(E_G)\otimes {\mathcal O}_X(-D) & \longrightarrow &\text{ad}(E_G)\otimes K_X\\
~\Big\downarrow{\mathcal R} && \Big\downarrow && ~\Big\downarrow\gamma\\
{\mathcal C}'_{\bullet} & : & \text{ad}(E_G) & \longrightarrow & \text{ad}(E_G)\otimes K_X(D)
\end{matrix}
$$
where the homomorphisms $$\text{ad}(E_G)\otimes{\mathcal O}_X(-D)\, \longrightarrow\, 
\text{ad}(E_G)\ \ \text{ and }\ \ \text{ad}(E_G)\otimes K_X\, \longrightarrow\, \text{ad}(E_G)\otimes 
K_X(D)$$ are the natural inclusions (recall that the divisor $D$ is effective, so
${\mathcal O}_X\, \hookrightarrow\, {\mathcal O}_X(D)$).

Serre duality gives that
\begin{equation}\label{sdh}
{\mathbb H}^1(({\mathcal C}')^\vee_{\bullet})\,=\, {\mathbb H}^1({\mathcal C}'_{\bullet})^*\, .
\end{equation}
Hence the above homomorphism ${\mathcal R}$ of complexes produces the following
homomorphism of hypercohomologies
\begin{equation}\label{er}
{\mathbb H}^1({\mathcal C}'_{\bullet})^*\,=\, {\mathbb H}^1(({\mathcal C}')^\vee_{\bullet})\,
\stackrel{{\mathcal R}_*}{\longrightarrow}\, {\mathbb H}^1({\mathcal C}'_{\bullet})\, ,
\end{equation}
where ${\mathcal R}_*$ is the homomorphism of hypercohomologies induced by $\mathcal R$;
the above isomorphism ${\mathbb H}^1({\mathcal C}'_{\bullet})^*\,=\, {\mathbb H}^1(({\mathcal C}')^\vee_{\bullet})$
is the one in \eqref{sdh}. The homomorphism ${\mathbb H}^1({\mathcal C}'_{\bullet})^*\,
{\longrightarrow}\, {\mathbb H}^1({\mathcal C}'_{\bullet})$ in \eqref{er}
is the homomorphism $P_{(E_G, \theta)}$ in the proposition that we are seeking.
\end{proof}

\section{Symplectic geometry of moduli of framed $G$--Higgs bundles}\label{sec:SP}

\subsection{Moduli space of framed $G$--Higgs bundles}\label{se5.1}

As before, for each point $x\, \in\, D$, fix a Zariski closed complex algebraic
proper subgroup $H_x$ of the complex reductive affine algebraic group $G$.

The topologically isomorphism classes of principal $G$--bundles on $X$ are parametrized by the
elements of the fundamental group $\pi_1(G)$ \cite{steenrod},
\cite[p.~186, Proposition 1.3(a)]{BLS}. Fix an element
$$
\nu\, \in\, \pi_1(G)\, .
$$

Let ${\mathcal M}_{H}(G)$ denote the moduli space of stable
$D$--twisted $G$--Higgs bundle on $X$ of the form
$(E_G,\, \theta)$, where
\begin{itemize}
\item $E_G$ is a holomorphic principal $G$--bundle on $X$ of topological type
$\nu$, and

\item $\theta\, \in\, H^0(X,\, \text{ad}(E_G)\otimes K_X(D))$.
\end{itemize}
(\cite{Si1, Si2, Nitin}).

Lemma \ref{lem-1} and Lemma \ref{lemn1} combine together to give the
following (see also Remark \ref{remn1}):

\begin{corollary}\label{cord0}
Assume that ${\rm genus}(X)\, \geq\, 1$. For any point $(E_G,\, \theta)\,\in\,
{\mathcal M}_{H}(G)$,
$$
T_{(E_G,\theta)}{\mathcal M}_{H}(G)\,=\, {\mathbb H}^1({\mathcal C}'_{\bullet})\, ,
$$
where ${\mathcal C}'_{\bullet}$ is the complex in \eqref{e10}.

The moduli space ${\mathcal M}_{H}(G)$
is a smooth orbifold of dimension $\dim G\cdot (2\cdot({\rm genus}(X)-1)
+n)+\dim Z({\mathfrak g})$, where $n\, =\,\# D$, and $Z({\mathfrak g})\, \subset\,
{\mathfrak g}\,:=\, {\rm Lie}({\mathfrak g})$ is the center of the Lie algebra.
\end{corollary}

Let ${\mathcal M}_{FH}(G)$ denote the moduli space of stable framed $G$--Higgs bundles
of topological type $\nu$ (\cite{Si1, Si2, Nitin, Ma,DMa,DG}). Let
\begin{equation}\label{vp}
\varphi\, :\, {\mathcal M}_{FH}(G)\, \longrightarrow\, {\mathcal M}_{H}(G)
\end{equation}
be the forgetful morphism that sends any triple $(E_G,\, \phi,\, \theta)$ to
$(E_G,\, \theta)$.

Define
\begin{equation}\label{int}
Z_{\mathfrak h}\, =\, (\bigcap_{x\in D} {\mathfrak h}_x)\cap Z({\mathfrak g})\, ,
\end{equation}
where ${\mathfrak h}_x$ as before denotes the Lie algebra of the subgroup $H_x$ of $G$.

\begin{proposition}\label{propdim}
Assume that ${\rm genus}(X)\, \geq\, 1$.
Let $(E_G,\, \phi,\, \theta)$ be a stable framed $G$--Higgs bundle. Let 
${\mathcal D}_{\bullet}$ be the complex in \eqref{e13} associated to $(E_G,\, \phi,\, \theta)$.
Then the following three hold:
\begin{enumerate}
\item ${\mathbb H}^0({\mathcal D}_{\bullet})\,=\, Z_{\mathfrak h}$, where $Z_{\mathfrak h}$
is defined in \eqref{int},

\item ${\mathbb H}^2({\mathcal D}_{\bullet})\,=\, Z^*_{\mathfrak h}$,

\item $\dim {\mathbb H}^1({\mathcal D}_{\bullet})\,=\, 2(\dim Z_{\mathfrak h}+
\dim G\cdot ({\rm genus}(X)-1+n) -\sum_{x\in D} \dim {\mathfrak h}_x)$, where $n\,=\,\# D$.
\end{enumerate}
\end{proposition}

\begin{proof}
Since ${\mathcal D}_{\bullet}$ is a sub-complex of ${\mathcal C}'_{\bullet}$ constructed
in \eqref{e10}, it follows that ${\mathbb H}^0({\mathcal D}_{\bullet})\, \subset\,
{\mathbb H}^0({\mathcal C}'_{\bullet})$. More precisely, from \eqref{e9a} we know that an
element
$$
v\, \in\, {\mathbb H}^0({\mathcal C}'_{\bullet})\, \subset\,
H^0({\mathcal C}'_0)\,=\, H^0(X,\, \text{ad}(E_G))
$$
lies in ${\mathbb H}^0({\mathcal D}_{\bullet})$ if and only if $v(x)\,\in\, {\mathcal H}_x$
for every $x\, \in\, D$. Now, from Lemma \ref{lemn1}
we know that ${\mathbb H}^0({\mathcal C}'_{\bullet})\,=\, Z({\mathfrak g})$. Combining these
it yields that ${\mathbb H}^0({\mathcal D}_{\bullet})\,=\, Z_{\mathfrak h}$. This proves
(1) in the proposition.

Using Serre duality and the isomorphism $\xi$ in \eqref{xi}, we have that
$$
{\mathbb H}^2({\mathcal D}_{\bullet})\,=\, {\mathbb H}^0({\mathcal D}^\vee_{\bullet})^*
\,=\, {\mathbb H}^0({\mathcal D}_{\bullet})^*\,=\, Z^*_{\mathfrak h}\, .
$$
This proves (2) in the proposition.

To prove (3), first note that from the long exact sequence of hypercohomologies associated
to the short exact sequence of complexes in \eqref{des} it follows immediately that
\begin{equation}\label{chif}
\dim {\mathbb H}^1({\mathcal D}_{\bullet})\,=\,
\dim {\mathbb H}^0({\mathcal D}_{\bullet})+ \dim {\mathbb H}^2({\mathcal D}_{\bullet})
- \chi(\text{ad}_\phi(E_G))+\chi(\text{ad}^n_\phi(E_G)\otimes
K_X(D))\, ;
\end{equation}
as before, $\chi$ denotes the Euler characteristic. Now, from \eqref{g1} we know that
$$
\chi(\text{ad}_\phi(E_G))\,=\, \chi(\text{ad}(E_G)(-D))+\sum_{x\in D} \dim {\mathcal H}_x\, .
$$
Hence $\chi(\text{ad}_\phi(E_G))\,=\, (\sum_{x\in D} \dim {\mathfrak h}_x)
-\dim G\cdot ({\rm genus}(X)-1+n)$.

Since $\text{ad}_\phi(E_G)^*\otimes K_X\,=\, \text{ad}^n_\phi(E_G)\otimes K_X(D)$ (see
Proposition \ref{prop1}(1)), using Serre duality, we have that
$$
\chi(\text{ad}^n_\phi(E_G)\otimes K_X(D))\,=\, \chi(\text{ad}_\phi(E_G)^*\otimes K_X)
\,=\,- \chi(\text{ad}_\phi(E_G))
$$
$$
=\, \dim G\cdot
({\rm genus}(X)-1+n)- \sum_{x\in D} \dim {\mathfrak h}_x\, .
$$
On the other hand, it was shown above that
$$\dim {\mathbb H}^2({\mathcal D}_{\bullet})\,=\, \dim {\mathbb H}^0({\mathcal D}_{\bullet})
\,=\, \dim Z_{\mathfrak h}\, .
$$
Combining these with \eqref{chif}, the third statement in the proposition follows.
\end{proof}

Lemma \ref{lem-2} and Proposition \ref{propdim} combine together to give the
following:

\begin{corollary}\label{cord}
Assume that ${\rm genus}(X)\, \geq\, 1$.
For any point $(E_G,\, \phi,\, \theta)\,\in\,
{\mathcal M}_{FH}(G)$,
$$
T_{(E_G,\phi,\theta)}{\mathcal M}_{FH}(G)\,=\, {\mathbb H}^1({\mathcal D}_{\bullet})\, ,
$$
where ${\mathcal D}_{\bullet}$ is the complex in \eqref{e13}.

The moduli space ${\mathcal M}_{FH}(G)$ is a smooth orbifold of dimension
$2(\dim Z_{\mathfrak h}+ \dim G\cdot ({\rm genus}(X)-1+n) -\sum_{x\in D}
\dim {\mathfrak h}_x)$.
\end{corollary}

Henceforth, we would always assume that ${\rm genus}(X)\, \geq\, 1$.

\subsection{Symplectic form on the moduli space}

Consider the symplectic form $\Phi_{(E_G, \phi, \theta)}$ in Lemma \ref{lem2}. In view
of Corollary \ref{cord}, this pointwise 
construction defines a holomorphic two-form on the moduli space ${\mathcal M}_{FH}(G)$.
This holomorphic two-form on ${\mathcal M}_{FH}(G)$ will be denoted by $\Phi$.

\begin{theorem}\label{thm1}
The above holomorphic form $\Phi$ on ${\mathcal M}_{FH}(G)$ is symplectic.
\end{theorem}

\begin{proof}
The form $\Phi$ is fiberwise nondegenerate by Lemma \ref{lem2}. So it suffices to
show that $\Phi$ is closed. 

Take any point $(E_G,\, \phi,\, \theta)\,\in\, {\mathcal M}_{FH}(G)$.
Corollary \ref{cord} says that
$$
T_{(E_G, \phi, \theta)} {\mathcal M}_{FH}(G)\,=\, {\mathbb H}^1({\mathcal D}_{\bullet})\, .
$$
Now consider the homomorphism $$\beta'_4\, :\, {\mathbb H}^1({\mathcal D}_{\bullet})\,
\longrightarrow\, H^1(X,\, \text{ad}_\phi(E_G))$$ in \eqref{g3}. In view of the first statement in
Proposition \ref{prop1}, Serre duality gives that
$$
H^1(X,\, \text{ad}_\phi(E_G))^*\,=\, H^0(X,\, \text{ad}_\phi(E_G)^*\otimes K_X)
\,=\, H^0(X,\, \text{ad}^n_\phi(E_G)\otimes K_X(D))\, .
$$
Now, since $\theta\, \in\, H^0(X,\, \text{ad}^n_\phi(E_G)\otimes K_X(D))$, we have the homomorphism
$$
\Psi_{(E_G, \phi, \theta)}\,:\, T_{(E_G, \phi, \theta)} {\mathcal M}_{FH}(G)\,=\,
{\mathbb H}^1({\mathcal D}_{\bullet})\, \longrightarrow\, \mathbb C\, ,\ \
w\, \longmapsto\, \theta(\beta'_4(w))\, .
$$
This pointwise construction of $\Psi_{(E_G, \phi, \theta)}$ produces a holomorphic
$1$-form on the moduli space ${\mathcal M}_{FH}(G)$. This holomorphic $1$-form on
${\mathcal M}_{FH}(G)$ will be denoted by $\Psi$.

The holomorphic $2$-form $d\Psi$ coincides with $\Phi$. Hence the form $\Phi$ is closed.
\end{proof}

\subsection{A Poisson map}

Take any $(E_G,\, \theta)\, \in\, {\mathcal M}_{H}(G)$. From Corollary \ref{cord0}
and \eqref{sdh} we know that
$$
T_{(E_G, \theta)}{\mathcal M}_{H}(G)\,=\, {\mathbb H}^1({\mathcal C}'_{\bullet})\ \ \text{ and }\ \
T^*_{(E_G, \theta)}{\mathcal M}_{H}(G)\,=\, {\mathbb H}^1(({\mathcal C}')^\vee_{\bullet})\, .
$$

The pointwise construction of the homomorphism $P_{(E_G, \theta)}$ in Proposition
\ref{prop2} produces a homomorphism
\begin{equation}\label{P}
P\, :\, T^*{\mathcal M}_{H}(G)\, \longrightarrow\, T{\mathcal M}_{H}(G)\, .
\end{equation}
This $P$ is a Poisson form on the moduli space ${\mathcal M}_{H}(G)$ \cite[p.~417, Theorem 4.6.3]{Bot}.

\begin{theorem}\label{thm2}
The forgetful function $\varphi$ in \eqref{vp} is a Poisson map.
\end{theorem}

\begin{proof}
Take any $\mathbf{z}\,:=\, (E_G,\, \phi,\, \theta)\,\in\, {\mathcal M}_{FH}(G)$.
Let $$\mathbf{y}\,:=\, \varphi(\mathbf{z})\,=\, (E_G,\, \theta) \,\in\, {\mathcal M}_{H}(G)$$
be its image under $\varphi$. Consider the differential of the map $\varphi$
\begin{equation}\label{df}
d\varphi(\mathbf{z})\, :\, T_{\mathbf{z}}{\mathcal M}_{FH}(G)\,\longrightarrow\,
T_{\mathbf{y}}{\mathcal M}_{H}(G)
\end{equation}
at the point $\mathbf{z}\,\in\, {\mathcal M}_{FH}(G)$. Let
\begin{equation}\label{dfd}
d\varphi(\mathbf{z})^*\, :\, T^*_{\mathbf{y}}{\mathcal M}_{H}(G)\,\longrightarrow\,
T^*_{\mathbf{z}}{\mathcal M}_{FH}(G)
\end{equation}
be the dual homomorphism.

In view of Corollary \ref{cord}, the isomorphism $(\Phi_{(E_G, \phi, \theta)})^{-1}$ in
\eqref{e15} is a homomorphism
\begin{equation}\label{Pi}
(\Phi_{(E_G, \phi, \theta)})^{-1}\,:\, T^*_{\mathbf{z}}{\mathcal M}_{FH}(G)\,
\stackrel{\sim}{\longrightarrow}\,
T_{\mathbf{z}}{\mathcal M}_{FH}(G)\, .
\end{equation}
Note that the homomorphism $(\Phi_{(E_G, \phi, \theta)})^{-1}$
in \eqref{Pi} defines the Poisson structure on ${\mathcal M}_{FH}(G)$
associated to the symplectic form $\Phi$ (see Theorem \ref{thm1}).

To prove the theorem, we need to show the following: For
every $w\, \in\, T^*_{\mathbf{y}}{\mathcal M}_{H}(G)$,
\begin{equation}\label{Sh}
d\varphi(\mathbf{z})\circ (\Phi_{(E_G, \phi, \theta)})^{-1}\circ d\varphi(\mathbf{z})^*(w)\,=\,
P(w)\, ,
\end{equation}
where $P$, $(\Phi_{(E_G, \phi, \theta)})^{-1}$, $d\varphi(\mathbf{z})$ and
$d\varphi(\mathbf{z})^*$ are the homomorphisms constructed in
\eqref{P}, \eqref{Pi}, \eqref{df} and \eqref{dfd} respectively, or in other words, the following
diagram of homomorphisms is commutative
$$
\begin{matrix}
T^*_{\mathbf{y}}{\mathcal M}_{H}(G) &\stackrel{P}{\longrightarrow} & T_{\mathbf{y}}{\mathcal M}_{H}(G)\\
d\varphi(\mathbf{z})^*\Big\downarrow\,\,\,\,\,\,\,\,\,\,\,\,\,\,\, & & d\varphi(\mathbf{z})
\Big\uparrow\,\,\,\,\,\,\,\,\\
T^*_{\mathbf{z}}{\mathcal M}_{FH}(G) &\stackrel{(\Phi_{(E_G, \phi, \theta)})^{-1}}{\longrightarrow}&
T_{\mathbf{z}}{\mathcal M}_{FH}(G)
\end{matrix}
$$
(see \cite[Section 4]{BLP}).

First consider the homomorphism $d\varphi(\mathbf{z})$ in \eqref{df}. Recall from
Corollary \ref{cord} and Corollary \ref{cord0}
respectively that $T_{\mathbf{z}} {\mathcal M}_{FH}(G)\,=\, {\mathbb H}^1({\mathcal D}_{\bullet})$
and $T_{\mathbf{y}} {\mathcal M}_{H}(G)\,=\, {\mathbb H}^1({\mathcal C}'_{\bullet})$. Now from the definition of
the forgetful map $\varphi$ in \eqref{vp} it follows immediately that
$d\varphi(\mathbf{z})$ coincides with the homomorphism of hypercohomologies
${\mathbb H}^1({\mathcal D}_{\bullet})\, \longrightarrow\, {\mathbb H}^1({\mathcal C}'_{\bullet})$ corresponding to
the following homomorphism of complexes:
$$
\begin{matrix}
{\mathcal D}_{\bullet} & : &
\text{ad}_\phi(E_G) & \longrightarrow &\text{ad}^n_\phi(E_G)\otimes K_X(D)\\
\Big\downarrow && \Big\downarrow && \Big\downarrow\\
{\mathcal C}'_{\bullet} & : & \text{ad}(E_G) & \longrightarrow & \text{ad}(E_G)\otimes K_X(D)
\end{matrix}
$$
where the homomorphisms $$\text{ad}_\phi(E_G)\, \longrightarrow\, \text{ad}(E_G)\ \ \text{ and }\ \
\text{ad}^n_\phi(E_G)\otimes K_X(D)\, \longrightarrow\,\text{ad}(E_G)\otimes K_X(D)$$ are the natural inclusions
(see \eqref{e9a} and \eqref{e12}).

Next consider the homomorphism $d\varphi(\mathbf{z})^*$ in \eqref{dfd}. Using
Corollary \ref{cord} and the isomorphism 
$\Phi_{(E_G, \phi, \theta)}$ in \eqref{e15} it follows that 
$T^*_{\mathbf{z}} {\mathcal M}_{FH}(G)\,=\, {\mathbb H}^1({\mathcal D}_{\bullet})$. On the 
other hand, we have $T^*_{\mathbf{y}}{\mathcal M}_{H}(G)\,=\, {\mathbb H}^1(({\mathcal 
C}')^\vee_{\bullet})$ (see Corollary \ref{cord0} and \eqref{sdh}); also, the complex $({\mathcal 
C}')^\vee_{\bullet}$ is realized as the complex in \eqref{ad}. Using these, the 
homomorphism $d\varphi(\mathbf{z})^*$ coincides with the homomorphism of hypercohomologies 
${\mathbb H}^1(({\mathcal C}')^\vee_{\bullet})\, \longrightarrow\, {\mathbb H}^1({\mathcal 
D}_{\bullet})$ corresponding to the following homomorphism of complexes:
$$
\begin{matrix}
({\mathcal C}')^\vee_{\bullet} & : &
\text{ad}(E_G)\otimes{\mathcal O}_X(-D) & \longrightarrow &\text{ad}(E_G)\otimes K_X\\
\Big\downarrow && \Big\downarrow && \Big\downarrow\\
{\mathcal D}_{\bullet} & : &
\text{ad}_\phi(E_G) & \longrightarrow &\text{ad}^n_\phi(E_G)\otimes K_X(D)
\end{matrix}
$$
where the homomorphisms $$\text{ad}(E_G)\otimes{\mathcal O}_X(-D)\, \longrightarrow\,
\text{ad}_\phi(E_G)\ \ \text{ and }\ \ \text{ad}(E_G)\otimes K_X\, \longrightarrow\, \text{ad}^n_\phi(E_G)\otimes K_X(D)$$
are the natural inclusions; see \eqref{g1} and Remark \ref{remn2}.

Consequently, the homomorphism $d\varphi(\mathbf{z})\circ (\Phi_{(E_G, \phi, \theta)})^{-1}\circ d\varphi(\mathbf{z})^*$
in \eqref{Sh} coincides with the homomorphism of hypercohomologies
\begin{equation}\label{eta}
\eta\, :\, {\mathbb H}^1(({\mathcal C}')^\vee_{\bullet})\, \longrightarrow\,
{\mathbb H}^1({\mathcal C}'_{\bullet})
\end{equation}
corresponding to
the following homomorphism of complexes:
$$
\begin{matrix}
({\mathcal C}')^\vee_{\bullet} & : &
\text{ad}(E_G)\otimes{\mathcal O}_X(-D) & \longrightarrow &\text{ad}(E_G)\otimes K_X\\
\Big\downarrow && \Big\downarrow && \Big\downarrow\\
{\mathcal C}'_{\bullet} & : & \text{ad}(E_G) & \longrightarrow & \text{ad}(E_G)\otimes K_X(D)
\end{matrix}
$$
where the homomorphisms $$\text{ad}(E_G)\otimes{\mathcal O}_X(-D)\, \longrightarrow\,
\text{ad}(E_G)\ \ \text{ and }\ \ \text{ad}(E_G)\otimes K_X\, \longrightarrow\, \text{ad}(E_G)\otimes K_X(D)$$
are the natural inclusions. But the homomorphism $\eta$ in \eqref{eta} evidently
coincides with the homomorphism $P_{(E_G, \theta)}$ constructed in Proposition \ref{prop2}.
Hence \eqref{Sh} is proved. As noted before, this completes the proof of the theorem.
\end{proof}

\section{The framework of Atiyah--Bott}\label{se6}

In this section we sketch an alternative construction of the symplectic form $\Phi$
in Theorem \ref{thm1} using the framework developed by Atiyah and Bott in \cite{AtBo}. This
framework was also employed by Hitchin in \cite{Hi0}.

Take any element $\nu\, \in\, \pi_1(G)$.
Fix a $C^\infty$ principal $G$--bundle $E^0_G$ on $X$ of topological type $\nu$.
Fix Zariski closed subgroups $H_x\, \subsetneq\, G$ for all $x\, \in\, D$.
\begin{enumerate}
\item{} Fix a framing $\phi_0$ on $E^0_G$ of type $\{H_x\}_{x\in D}$, so $\phi_0(x)$
is an element of the quotient space
$(E^0_G)_x/H_x$ for every $x\, \in\, D$.

\item{} The space of all holomorphic structures on the principal $G$--bundle $E^0_G$
is an affine space for the vector space $C^\infty(X,\, \text{ad}(E^0_G)\otimes\Omega^{0,1}_X)$.
Fix a holomorphic structure on the
$C^\infty$ principal $G$--bundle $E^0_G$; the resulting holomorphic principal $G$--bundle
will be denoted by ${\mathcal E}^0_G$.

\item Fix a Higgs field $\theta_0$ on the framed holomorphic
principal $G$--bundle $({\mathcal E}^0_G,\, \phi_0)$.
\end{enumerate}
As done in \eqref{e6}, let
\begin{equation}\label{s0}
\widehat{\sigma}_0\, \in\, C^\infty(X,\, \text{Sym}^2(\text{ad}(E^0_G)))
\end{equation}
be the fiberwise nondegenerate symmetric bilinear form defined by $\sigma$ in \eqref{e5};
the subscript ``$0$'' in ``$\widehat{\sigma}_0$'' is to emphasize the fact that this pairing
is on a fixed vector bundle $\text{ad}(E^0_G)$.

Recall the constructions of $\text{ad}_\phi(E_G)$ and $\text{ad}^n_\phi(E_G)$, done in
\eqref{e9a} and \eqref{e12} respectively, for a framed principal $G$--bundle $(E_G,\, \phi)$.
Substituting the above framed principal $G$--bundle $({\mathcal E}^0_G,\, \phi_0)$ in place
of $(E_G,\, \phi)$ in the constructions done in \eqref{e9a} and \eqref{e12}, we get holomorphic
vector bundles $\text{ad}_{\phi_0}({\mathcal E}^0_G)$ and $\text{ad}^n_{\phi_0}({\mathcal E}^0_G)$
respectively.

Let ${\mathcal V}^{0,1}$ denote the space of all $C^\infty$ sections of the vector bundle
$$\text{ad}_{\phi_0}({\mathcal E}^0_G)\otimes \overline{K}_X\,=\,
\text{ad}_{\phi_0}({\mathcal E}^0_G)\otimes \Omega^{0,1}_X\, .$$
The space of all $C^\infty$ sections of the vector bundle
$\text{ad}^n_{\phi_0}({\mathcal E}^0_G)\otimes K_X(D)$ will be denoted by ${\mathcal V}^{1,0}$.
Now construct the direct sum of vector spaces
\begin{equation}\label{cw}
{\mathcal W} \, :=\, {\mathcal V}^{0,1}\oplus {\mathcal V}^{1,0}\, .
\end{equation}
Given any $v\, \in\, {\mathcal V}^{0,1}$, we get a framed holomorphic principal $G$--bundle
$({\mathcal E}^v_G,\, \phi_v)$ on $X$. To clarify, the underlying $C^\infty$ framed principal
$G$--bundle for $({\mathcal E}^v_G,\, \phi_v)$ is $(E^0_G,\, \phi_0)$, and the almost complex
structures of ${\mathcal E}^v_G$ and ${\mathcal E}^0_G$ differ by $v$; as mentioned before, the
space of all holomorphic structures on $E^0_G$ is an affine space for $C^\infty(X,\,
\text{ad}(E^0_G)\otimes\Omega^{0,1}_X)$. It may be mentioned that
these conditions uniquely determine ${\mathcal E}^v_G$. Also, note that the framing
$\phi_v$ coincides with $\phi_0$ using the $C^\infty$ identification
between $E^0_G$ and ${\mathcal E}^v_G$. Now
consider the Dolbeault operator for the holomorphic vector bundle $\text{ad}({\mathcal E}^v_G)$;
we shall denote it by $\overline{\partial}^v_1$.
This Dolbeault operator $\overline{\partial}^v_1$ and the Dolbeault operator for the holomorphic
line bundle $K_X(D)$ together define the Dolbeault operator for the holomorphic vector bundle 
$\text{ad}({\mathcal E}^v_G)\otimes
K_X(D)$. This Dolbeault operator for $\text{ad}({\mathcal E}^v_G)\otimes K_X(D)$ will be denoted by
$\overline{\partial}^v$. Let
\begin{equation}\label{cw0}
{\mathcal W}^0\, \subset\, {\mathcal W}
\end{equation}
be the subset of the direct sum in \eqref{cw} consisting of all $(v,\, w)\, \in\, 
{\mathcal V}^{0,1}\oplus {\mathcal V}^{1,0}$ such that
\begin{equation}\label{abe}
\overline{\partial}^v(w)\,=\, 0\, .
\end{equation}
Therefore,
for any $(v,\, w)\, \in\, {\mathcal W}^0$, the section $w$ is a (holomorphic) Higgs field on the
framed holomorphic principal $G$--bundle $({\mathcal E}^v_G,\, \phi_v)$.

We shall now construct a complex $1$-form on ${\mathcal W}$. For any
$$
(v,\, w)\, \in\, {\mathcal V}^{0,1}\oplus {\mathcal V}^{1,0}\, ,
$$
we have $\widehat{\sigma}_0(v,\, w)\, \in\, C^\infty(X,\, \Omega^{1,1}_X)$, where
$\widehat{\sigma}_0$ is constructed in \eqref{s0}. Note that
while $w$ may have a pole over $D$ as a section of $\text{ad}(E^0_G)$, the
pairing $\widehat{\sigma}_0(v,\, w)$ does not have a pole as a section of $\Omega^{1,1}_X$, because
the image of $\text{ad}_{\phi_0}({\mathcal E}^0_G)_x$ in $\text{ad}({\mathcal E}^0_G)_x$ annihilates
the image of $\text{ad}^n_{\phi_0}({\mathcal E}^0_G)_x$ in $\text{ad}({\mathcal E}^0_G)_x$
for the nondegenerate bilinear form $\widehat{\sigma}_0(x)$ on
$\text{ad}({\mathcal E}^0_G)_x$ for all $x\, \in\, D$. (To see this,
recall from \eqref{e9a} that the image
of $\text{ad}_\phi(E_G)_x$ in $\text{ad}(E_G)_x$ is ${\mathcal H}_x$, while from \eqref{e12} we know that
the image of $\text{ad}^n_\phi(E_G)_x$ in $\text{ad}(E_G)_x$ is ${\mathcal H}^\perp_x$.) Let
\begin{equation}\label{Pp}
\Psi'_0\, \in\, H^0({\mathcal W}, \, \Omega^1_{\mathcal W})
\end{equation}
be the holomorphic $1$-form on $\mathcal W$ defined by
$$
\Psi'_0(v,\, w)(v_1,\, w_1)\, \longmapsto\, \int_X \widehat{\sigma}_0(v_1,\, w)\, \in\, \mathbb C\, ,
$$
for all $(v,\, w)\, \in\, {\mathcal W}$ and $(v_1,\, w_1)\, \in\, T_{(v, w)} {\mathcal W}
\,=\, \mathcal W$; here we are using the fact that
the tangent space $T_{(v, w)}{\mathcal W}$ is canonically identified
with $\mathcal W$ itself as ${\mathcal W}$ is a complex vector space. Note that
using the element of ${\mathcal W}^*$ defined by
$$
(v,\, w)\, \longmapsto \int_X \widehat{\sigma}_0(v,\, w)\, \in\, \mathbb C\, ,
$$
the vector space ${\mathcal V}^{1,0}$ is embedded into the dual vector space
$({\mathcal V}^{0,1})^*$. This embedding produces a holomorphic embedding of
${\mathcal W}$ inside the holomorphic cotangent bundle $(T{\mathcal V}^{0,1})^*$. Using
this embedding, the form $\Psi'_0$ in \eqref{Pp} is the restriction of the Liouville $1$-form on
the holomorphic cotangent bundle $(T{\mathcal V}^{0,1})^*$.

The de Rham differential
\begin{equation}\label{drd}
d\Psi'_0\, =:\, \Phi'_0
\end{equation}
has the following expression: For any $$(v,\, w)\, \in\, {\mathcal W}$$ and any two
tangent vectors $(v_1,\, w_1),\, (v_2,\, w_2)\, \in\, T_{(v, w)} {\mathcal W}$,
$$
\Phi'_0(v,\, w) ((v_1, w_1), (v_2, w_2))\,=\, \int_X
(\widehat{\sigma}_0(v_2, w_1)-\widehat{\sigma}_0(v_1, w_2))\, .
$$
Let $\Psi_0$ and $\Phi_0$ be the restrictions to ${\mathcal W}^0$ (see \eqref{cw0})
of the above defined differential forms $\Psi'_0$ and $\Phi'_0$ respectively.

Let $\mathcal G$ denote the group of all $C^\infty$ automorphisms of the principal 
$G$--bundle $E^0_G$ preserving the framing $\phi_0$. It is straight-forward to check that 
the Lie algebra of $\mathcal G$ is $C^\infty(X,\, \text{ad}_{\phi_0}({\mathcal E}^0_G))$. 
This group $\mathcal G$ has a natural action on ${\mathcal W}$; this action of $\mathcal G$ 
on ${\mathcal W}$ evidently preserves the subset ${\mathcal W}^0$ defined in \eqref{cw0}. 
The $1$-form $\Psi'_0$ on ${\mathcal W}$ is evidently preserved by the action of $\mathcal 
G$ on ${\mathcal W}$, because $\widehat{\sigma}_0$ is preserved under the action of 
$\mathcal G$ on $\text{ad}(E^0_G)$ induced by the action of $\mathcal G$ on the principal 
$G$--bundle $E^0_G$. Consequently, the action of the group $\mathcal G$ on ${\mathcal W}^0$ 
preserves the form $\Psi_0$. The de Rham differential $d\Psi'_0$ is preserved by the action 
of $\mathcal G$ on ${\mathcal W}$, because $\Psi'_0$ is preserved by the action of $\mathcal 
G$ on ${\mathcal W}$. Therefore, the $2$-form $\Phi_0\,=\, (d\Psi'_0)\vert_{{\mathcal W}^0}$ 
is also preserved by the action of $\mathcal G$ on ${\mathcal W}^0$.

Take any element $(v,\, w)\, \in\, {\mathcal W}^0$. As before, $\overline{\partial}^v_1$ and 
$\overline{\partial}^v$ denote the Dolbeault operators for $\text{ad}({\mathcal E}^v_G)$ and 
$\text{ad}({\mathcal E}^v_G)\otimes K_X(D)$ respectively. Take any section
$\beta\,\in\, C^\infty(X,\, \text{ad}_{\phi_0}({\mathcal E}^0_G))$. Now we have
$$
\int_X \widehat{\sigma}_0(\overline{\partial}^v_1(\beta), w)\,=\,
- \int_X \widehat{\sigma}_0(\beta, \overline{\partial}^v(w))\,=\,0\, ,
$$
because $\overline{\partial}^v(w)\,=\, 0$ (see \eqref{abe}). As a consequence of it, the $1$-form $\Psi_0$
on ${\mathcal W}^0$ descends 
under the action of $\mathcal G$ on ${\mathcal W}^0$. Hence $\Phi_0\,=\, d\Psi_0$ also 
descends under the action of $\mathcal G$ on ${\mathcal W}^0$. The descent of $\Psi_0$ 
corresponds to the form $\Psi$ in the proof of Theorem \ref{thm1}, while the descent of 
$\Phi_0$ corresponds to the form $\Phi$ in Theorem \ref{thm1}. From \eqref{drd} if follows that
$\Phi\,=\, d\Psi$.

\section{The Hitchin system: cameral data for framed $G$--Higgs bundles}\label{sec:cameral}

In this section we shall describe the Hitchin integrable system for framed $G$--Higgs bundles.
We will assume that $H_x\,=\, e$ for all $x\, \in\, D$ as it is quite similar to the general case
while being simpler to present; some remarks on the general case are
included for the sake of completeness.

For any holomorphic Poisson manifold $(M,\, \sharp)$, we denote by
$\{\cdot,\,\cdot\}_{\sharp}$ the associated Poisson bracket on $\mathcal{O}_{M}$, i.e., $\{f,\,g\}_{\sharp}
\,=\,\pi^{\sharp}(df,\,dg)$ where $\pi^{\sharp}\in\Gamma(\bigwedge^{2}(TM))$ is the Poisson bi-vector.

A symplectic structure $\omega$ on $M$ also defines a Poisson bracket on $\mathcal{O}_{M}$ by assigning to 
$(f,\,g)\,\in\,\mathcal{O}_{M}\times \mathcal{O}_{M}$ the function $\{f,\,g\}_{\omega}\,=\,
\omega(X_{g},\,X_{f})$, where $X_{f}$ and $X_{g}$ are 
the Hamiltonian vector fields defined by $f$ and $g$ with respect to $\omega$.

Two functions $f,\,g\in \mathcal{O}_{M}$ are said to Poisson commute if
$$
\{f,\, g\}\,=\,0\, .
$$
An \emph{algebraically completely integrable system} on $M$ consists
of functions $f_1,\, \cdots ,\,f_d\,\in\, \mathcal{O}_{M}$ with $d\,=\, \frac{1}{2}\dim M$,
such that
\begin{itemize}
\item $\{f_i,\, f_j\}\,=\,0$ for all $1\, \leq\, i,\, j\, \leq\, d$,

\item the corresponding Hamiltonian vector fields $X_{f_1},\, \cdots, \, X_{f_d}$ are linearly independent
at the general point, and

\item the generic fiber of the map $(f_1,\cdots,\, f_d) \,:\, M\,\longrightarrow\,
\mathbb{C}^d$ is a open set in an abelian variety
such that the vector fields $X_{f_1},\, \cdots, \, X_{f_d}$ are linear on it.
\end{itemize}

\subsection{Recollection: the Hitchin system for Higgs bundles}

Fix a Borel subgroup $B\,\subset\,G$ and a Cartan subgroup $T\,\subset\,B$. Let
$\mathfrak{t}\,\subset\, \mathfrak{b}$ be the Lie algebras of $T$ and $B$. The Weyl group
$N_G(T)/T$, where $N_G(T)$ is the normalizer of $T$ in $G$, will be denoted by $W$.

Consider the Chevalley morphism 
\begin{equation}\label{eq:Chevalley_morph}
\chi\,:\, \mathfrak{g}\,\longrightarrow\,\mathfrak{t}/W
\end{equation}
constructed using the isomorphism $\mathfrak{t}/W\,\cong\, \mathfrak{g}\sslash G\,:=\,\mathrm{Spec}(\mathbb{C}[\mathfrak{g}]^G)$. Since
$\mathbb{C}[\mathfrak{g}]^G$ is generated by homogeneous polynomials of degrees $d_1,\,\cdots,\,
d_r$, where $r\,=\,\mathrm{rank}(G)$, it admits a graded $\mathbb{C}^\times$ action. The
induced action on $\mathfrak{g}\sslash G$ turns
$\chi$ into a $\mathbb{C}^\times$--equivariant morphism. This, using the $G$--invariance
property of the morphism \eqref{eq:Chevalley_morph}, induces a map:
\begin{equation}\label{eq:Hitchin_map_H} 
h\,:\,\mathcal{M}_{H}(G)\,\longrightarrow\, \mathcal{B}\,:=\,H^0(X,\,\mathfrak{t}\otimes K_X(D)/W)
\end{equation}
given by 
$$
h(E,\,\theta)(x)\,=\, \chi(\theta(x))\, .
$$
Alternatively, the choice of $r$ generators $p_1,\,\cdots,\, p_r$ of
$\mathbb{C}[\mathfrak{g}]^G$ of degrees $\deg p_i\,=\,d_i$, $i\,=\,1,\,\cdots, \,r$, induces
an isomorphism
$$
H^0(X,\,\mathfrak{t}\otimes K_X(D)/W)\,\cong\, \bigoplus_{i=1}^{r}H^{0}(X,\,
K^{d_i}_X(d_iD))
$$
under which $h$ can be described as
\begin{equation}\label{eq:alt_hm}
h(E,\theta)(x)\,=\, \left(p_1(\theta(x)),\,\cdots, \,p_r(\theta(x))\right)\, ,\qquad
p_i(\theta)\,\in\, H^0(X,\,K^{d_i}_X(d_i D)).
\end{equation}
The dimension of the vector space $\mathcal{B}$ thus is 
\begin{eqnarray}\label{dimH}\nonumber
N\,:=\,\dim \mathcal{B}\,&=&\sum_{i=1}^r(d_i(2g(X)-2+n)-g(X)+1)\,=\,(g(X)-1)\sum_{i=1}^r(2d_i-1)+n\sum_{i=1}^rd_i\\
&=&(g(X)-1)\dim G+n\cdot\dim B\, ,
\end{eqnarray}
where $g(X)$ is the genus of $X$.

Given any $b\,\in\,\mathcal{B}$, we define the corresponding \emph{cameral cover} as the curve ${X}_b$ given by the commutative diagram:
\begin{equation}\label{eq:cam_cov}
\xymatrix{	{X}_b\ar[r]\ar[d]_{\pi_b}&\mathfrak{t}\otimes K_X(D)\ar[d]\\
	X\ar[r]_{\hspace*{-25pt}b}&\mathfrak{t}\otimes K_X(D)/W}
\end{equation}
Consider the generic locus $\mathcal{B}^{sm}$ corresponding to sections whose associated cameral cover in \eqref{eq:cam_cov} is smooth. Then, by \cite[Proposition 4.7.7]{Ngo},
the inverse image $h^{-1}(b)$ is contained in the locus of $\mathcal{M}_H(G)$ consisting of Higgs bundles
for which the Higgs field $\theta$ is regular at every point, meaning that the orbit of $\theta(x)$ is maximal
dimensional. Moreover, by \cite[Corollary 17.8]{DG}, the choice of a point in the fiber induces an isomorphism
\begin{equation}\label{abel1}
h^{-1}(b)\,\cong\, H^1({X}_b,\, T)^W\, ,
\end{equation}
where the action of $W$ on a principal $T$-bundle $P\,\longrightarrow\, X_b$ is given by
$$
w\cdot P\,=\,(w^*P\times_w T)\otimes \mathcal{R}_w\, .
$$
In the above, $\mathcal{R}_w$ is a principal $T$--bundle naturally associated to the ramification
divisor of $w$ (cf. \cite[\S~5]{DG}). Moreover, there exists a group scheme $J\,\longrightarrow
\,X\times\mathcal{B}$ such that 
$$h^{-1}(b)\,\cong\,H^1(X,\,J_b),$$
where $J_b\,=\,J|_{X\times\{b\}}$. In
other words, the automorphism group of elements of the Hitchin fiber $h^{-1}(b)$ (seen as torsors over $X_b$) descends to $J_b\, \longrightarrow X$. 

In the language of stacks, let 
$\mathbf{M}_H(G)$ be the stack of $G$--Higgs bundles. In a similar way as done in \eqref{eq:Hitchin_map_H} we may define a stacky Hitchin map by:
\begin{equation}\label{eq:stacky_hm}
\begin{array}{ccc}
	\mathbf{h}\,\colon\,\mathbf{M}_H(G)&\longrightarrow& \mathcal{B}\\
	(E,\,\phi)&\longmapsto&\chi(\phi)\, ,
\end{array}
\end{equation}
where $\chi$ is the Chevalley morphism \eqref{eq:Chevalley_morph}. 

Consider the Picard stack $\mathcal{P}\,\longrightarrow 
\,\mathcal{B}$ of principal $J$--bundles. Then, $\mathbf{M}_H(G)|_{\mathcal{B}^{sm}}$ is a torsor over $\mathcal{P}|_{\mathcal{B}^{sm}}$ 
relative to $\mathcal{B}^{sm}$. In particular, if $b\in\mathcal{B}^{sm}$, 
we have an isomorphism
\begin{equation}\label{abel2}
\mathbf{h}^{-1}(b)\cong\mathcal{P}_b,
\end{equation}
determined by a choice of an element of the fiber. 

\begin{lemma}\label{lm:dim_fibers}
\begin{equation}\label{eq:dim_fiber}
\dim h^{-1}(b)\,=\,(g(X)-1)\dim G+n(\dim B-\dim T)+\dim Z(G)\, .
\end{equation}
\end{lemma}

\begin{proof}
By Lemma \ref{lemn1} we have that $\mathcal{M}_H(G)=\mathbf{M}_H(G)\fatslash Z(G)$ (where the symbol $\fatslash$ denotes rigidification 
\cite[Appendix A]{AOV}). So it follows that
$$\dim h^{-1}(b)=\dim\mathcal{P}_b+\dim Z(G)=(g(X)-1)\dim G+n(\dim B-\dim T)+\dim Z(G)\,,$$
where the second equality is \cite[Corollary 4.13.3]{Ngo}.
\end{proof}

The above facts about abelianization of generic fibers (\ref{abel1}) and (\ref{abel2}), the dimensions in Lemma \ref{lm:dim_fibers} and Corollary \ref{cord0}, together with the following proposition prove that the Hitchin map is an algebraically completely integrable system on the Poisson variety $\mathcal{M}_H(G)$.

\begin{proposition}\label{PC-GHiggs}
Let $P$ be the Poisson structure on $\mathcal{M}_{H}(G)$ described in (\ref{P}) and $\{\cdot,\cdot\}_{P}$ its associated Poisson bracket. The $N$ functions on $\mathcal{M}_{H}(G)$ provided by the Hitchin system $h$ in (\ref{eq:Hitchin_map_H}) Poisson-commute with respect to $\{\cdot,\cdot\}_{P}$.
\end{proposition}

\begin{proof}
This follows from the results in \cite[Theorem 8.5, Remark 8.6]{Ma} and \cite[Section 5]{DMa}.
\end{proof}

\subsection{The Hitchin morphism for framed $G$-- Higgs bundles}

Now consider the morphism
\begin{equation}\label{eq:Hitchin_map_FH} 
h_{FH}\,:\,\mathcal{M}_{FH}(G)\,\longrightarrow\, \mathcal{B}
\end{equation}
defined by the commutative diagram
\begin{equation}\label{eq:commuting_hms}
\xymatrix{
	\mathcal{M}_{FH}(G)\ar[d]_\varphi\ar[dr]^{h_{FH}}&\\
	\mathcal{M}_H(G)\ar[r]^{\,\,\,h}&	\mathcal{B}
	}
\end{equation}
where $\varphi$ is defined in \eqref{vp} and $h$ is in \eqref{eq:Hitchin_map_H}. 

\begin{remark}\label{rk:alt_hmf}
By commutativity of \eqref{eq:commuting_hms}, it turns out that $h_{FM}$ can also be expressed in terms of invariant
polynomials as in \eqref{eq:alt_hm}.
\end{remark}

Note that Proposition \ref{PC-GHiggs} and Theorem \ref{thm2} together give the following.

\begin{corollary}\label{cor:HF-Poissonc}
Let $\Phi$ be the holomorphic symplectic form on $\mathcal{M}_{FH}(G)$ and $\{\cdot,\cdot\}_{\Phi}$ it associated Poisson
bracket. The $N$ functions in $h_{FH}$ Poisson commute with respect to $\{\cdot,\cdot\}_{\Phi}$
\end{corollary}

Let $Z(G)$ denote the center of $G$.

\begin{proposition}\label{prop phi is torsor}
The forgetful map $\varphi$ in \eqref{vp} makes ${\mathcal M}_{FH}(G)$ a torsor over the orbifold
${\mathcal M}_{H}(G)$ for the group $(\prod_{x\in D} G)/Z(G)\,=\, G^n/Z(G)$, where $n\,=\, \# D$ and $Z(G)$
is embedded diagonally in $G^n$.
\end{proposition}

\begin{proof}
take any $(E_G,\, \theta)\, \in\, {\mathcal M}_{H}(G)$.
The group $Z(G)$ is a subgroup of the
group $\text{Aut}(E_G, \theta)$ parametrizing all holomorphic automorphisms of the $D$-twisted
$G$--Higgs bundle $(E_G,\, \theta)$. In fact $Z(G)$ is a normal
subgroup of $\text{Aut}(E_G, \theta)$ such that quotient $\text{Aut}(E_G, \theta)/Z(G)$ coincides with the inertia group
of the orbifold point $(E_G,\, \theta)\, \in\, {\mathcal M}_{H}(G)$.

Now consider
$$
{\mathcal F}(E_G)\,=\, \prod_{x\in D} (E_G)_x/H_x\, =\, \prod_{x\in D} (E_G)_x
$$
constructed in \eqref{e3}. From the action of $G$ on $(E_G)_x$, $x\, \in\, D$, we get an action of
$G^n\,=\, \prod_{x\in D}G$ on $\prod_{x\in D} (E_G)_x$.
Consider $Z(G)$ embedded diagonally in $\prod_{x\in D}G$.
The action of this subgroup $Z(G)\, \subset\, \prod_{x\in D}G$ on $\prod_{x\in D} (E_G)_x$ factors through the
tautological action of $\text{Aut}(E_G,\theta)$ on $\prod_{x\in D} (E_G)_x$.

On the other hand, the inverse image $\varphi^{-1}(E_G, \theta)\, \subset\,
{\mathcal M}_{FH}(G)$ is evidently identified with ${\mathcal F}(E_G)/\text{Aut}(E_G,\theta)$. This proves that
the orbifold ${\mathcal M}_{FH}(G)$ is a torsor over
${\mathcal M}_{H}(G)$ for the group $(\prod_{x\in D} G)/Z(G)\,=\, G^n/Z(G)$.
\end{proof}

From Proposition \ref{prop phi is torsor} a description of the Hitchin fibers is obtained.

\begin{corollary}\label{cor:torsor-fibres}
The forgetful morphism induces a $G^n/Z(G)$-torsor structure 
$$
{h}_{FH}^{-1}(b)\longrightarrow{h}^{-1}(b).
$$
\end{corollary} 

In particular, the Hitchin system is not abelianizable, thus neither is it algebraically 
completely integrable. Note also that the number of Poisson commuting functions provided by 
$h_{FH}$ is less than half of the dimension of $\mathcal{M}_{FH}(G)$. We next define a 
maximally abelianizable subsystem such that its dimension doubles the number of Poisson 
commuting functions. In order to do that, we need to introduce some more notation.

Consider the stack of stable framed Higgs bundles $\mathbf{M}_{FH}(G)$. Forgetting the frame induces a $G^n$-torsor
$\Phi:\mathbf{M}_{FH}(G)\longrightarrow \mathbf{M}_H(G)$ by Proposition \ref{prop phi is torsor}. Now, the Hitchin map
in \eqref{eq:Hitchin_map_FH} also admits a stacky version $\mathbf{h}_{FH}$ defined by the commutative diagram:
\begin{equation}\label{eq:stacky_framed_hm}
\xymatrix{
\mathbf{M}_{FH}(G)\ar[d]_\Phi\ar[dr]^{\mathbf{h}_{FH}}&\\
\mathbf{M}_H(G)\ar[r]^{\,\,\,\mathbf{h}}& \mathcal{B},\,
}
\end{equation}
where $\Phi$ is the forgetful morphism and $\mathbf{h}$ is defined in \eqref{eq:stacky_hm}. 
Note that by Proposition \ref{propdim} we have $\mathcal{M}_{FH}(G)\,=\,\mathbf{M}_{FH}(G)\fatslash Z(G)$, so the
following commutative diagram is obtained
\begin{equation}\label{eq:commuting_rigidifications}
\xymatrix{
\mathbf{M}_{FH}(G)\ar[rr]\ar[dd]_\Phi\ar[dr]^{\mathbf{h}_{FH}}&&\mathcal{M}_{FH}(G)\ar[dd]^\varphi\ar[dl]_{h_{FH}}\\
& \mathcal{B}&\\
\mathbf{M}_H(G)\ar[rr]\ar[ur]^{\mathbf{h}}&&\mathcal{M}_H(G)\ar[ul]_h,\,
}
\end{equation}
where the horizontal arrows are $Z(G)$--torsors defined via rigidification.

\begin{lemma} 
The forgetful morphism	$\mathbf{h}_{FH}^{-1}(b)\,\longrightarrow\, \mathbf{h}^{-1}(b) $ induces a $G^n$ torsor structure.
\end{lemma}

\begin{proof}
Commutativity of \eqref{eq:stacky_framed_hm} implies that $\Phi$ takes fibers of $\mathbf{h}_{FH}$ to fibers
of $\mathbf{h}$. 
The rest follows as in the proof of Proposition \ref{prop phi is torsor}, after incorporating the observation that 
quotienting by automorphisms of the base is not necessary when working with stacks.
\end{proof}

\subsection{Relatively framed Higgs bundles}

In this section we produce a subsystem of the Hitchin system \eqref{eq:Hitchin_map_FH} which is an algebraically completely integrable system. 

Consider $\mathcal{B}_{nr}^{sm}\subset \mathcal{B}^{sm}$, the subset of smooth cameral covers
unramified over $D$. Over this we consider the stack $\mathcal{P}_{FH}$ of $J$ principal bundles
with a $W$ and $T$--equivariant framing over ${D}_{\mathcal{B}}\,=\,
(D\times\mathcal{B})\times_{\mathfrak{t}\otimes K_X(D)/W} \mathfrak{t}\otimes K_X(D)$. If
$D_b\,=\,D_{\mathcal{B}}|_{b}$, then
equivariance of $\delta\,:\,P|_{{D}_b}\,\cong\,{{D}_b}\times T$ is given by
\begin{equation}\label{eq:equivar}\delta_{w^{-1}x}=w^{-1}\circ\delta_x
\end{equation}
where
$$\delta_x\,:\,P_{x}\,\stackrel{\sim}{\longrightarrow}\, T
$$
is the frame at a point $x\,\in \,D_b$ and $w^{-1}\,:\,T\,\longrightarrow\, T$ is the usual action.

By the following proposition, $\mathcal{P}_{FH}$ is an abelian group stack relative to
$\mathcal{B}_{nr}^{sm}$.

\begin{proposition}\label{prop_relative Picard}
The forgetful morphism 
\begin{equation}\label{eq:morphism_Picard_stacks}
\mathcal{P}_{FH}\,\longrightarrow\, \mathcal{P}
\end{equation}
induces a $T^n$ torsor structure.
\end{proposition}

\begin{proof}
Let $(E,\,\theta,\,\delta_i)\,\in\,\mathcal{P}_{F,b}(X)$, $i\,=\,1,\,2$. Then, the equivariance condition \eqref{eq:equivar} implies 
that $\delta_x$ commutes with all the automorphisms of $(E,\,\theta)$ inside $\mathbf{h}^{-1}(b)$. Hence one obtains a $J|_D$ torsor. 
But since by assumption $D_b\,\longrightarrow\, D$ is unramified, this is a $T^n$--torsor. See \cite[\S~2.5]{Ngo}.
\end{proof}

\begin{theorem}\label{thm:fibers}
The equivalence $\mathbf{h}^{-1}(b)\cong\mathcal{P}_b$ induces a faithful morphism
$$
\mathcal{P}_{F,b}\,\hookrightarrow\,\mathbf{h}_{FH}^{-1}(b)\, .
$$
\end{theorem}

\begin{proof}
Let $(E_G,\,\theta,\,\delta)\,\in\, \mathbf{h}_{FH}^{-1}(b)$, and let $P\,\in\,\mathcal{P}_b(X)$ be the object corresponding to
$(E_G,\,\theta)$ via the equivalence $\mathbf{h}_{FH}^{-1}(b)\,\cong\,\mathcal{P}_b(X)$. Since $X_b$ is not ramified over $D$, the
equivariance conditions on $P$ and $\delta$, together with \cite[Proposition 7.5]{LP} imply that $P|_{D_b}$ and $\delta$ descend
to $E|_D$ and a trivialization $E|_D\,\cong\, D\times N$, where $N$ is the normalizer of $T$ in $G$.
	
Since all the steps are functorial, this defines a morphism of stacks. Faithfulness follows from the fact that these are categories 
fibered in groupoids and that the action of $\mathcal{P}_{FH}$ on $\mathbf{h}^{-1}_{FH}(b)$ is compatible with the torsor structures 
over $\mathcal{P}$ and $\mathbf{h}^{-1}(b)$ respectively.
\end{proof}

We define the sub-stack of \emph{relatively framed Higgs bundles} as 
\begin{equation}\label{eq:diag_FH}
\mathbf{M}^{\Delta}_{FH}(G)\, :=\, \mathrm{Im}\left(\mathcal{P}_{F,b}\hookrightarrow	\mathbf{h}_{FH}^{-1}(b)\right).
\end{equation}
 Let $\mathcal{M}_{FH}^\Delta(G):=\mathbf{M}^{\Delta}_{FH}(G)\fatslash Z(G)$. Consider the restriction of the Hitchin map
 \begin{equation}\label{eq:diag_Hitchin}
h^\Delta_{FH}\,:\,\mathcal{M}_{FH}^\Delta(G)\,\longrightarrow\, \mathcal{B}_{nr}^{sm}\, .
 \end{equation}

\begin{corollary}\label{cor:subsystem}
The fibers of $h^{\Delta}_{FH}$ are $N$-dimensional semiabelian varieties. Therefore the moduli space $\mathcal{M}_{FH}^\Delta(G)$ is maximally abelianizable. Moreover, the $N$-functions $(h_1,\,\cdots,\,h_N)$ obtained by identifying
$\mathcal{B}\,\cong\,\mathbb{C}^N$ and $h_{FH}\,=\,(h_1,\,\cdots,\, h_N)$ are in involution.
\end{corollary}

\begin{proof}
We have a commutative diagram
$$
\xymatrix{
\mathbf{h}^{-1}_{FH}(b)\ar[r]\ar[d]&
{h}^{-1}_{FH}(b)\ar[d]\\
\mathbf{h}^{-1}(b)\ar[r]&{h}^{-1}_{H}(b),
}
$$
which by Theorem \ref{thm:fibers} implies that there is a short exact sequence
$$
0\longrightarrow T^n/Z(G)\longrightarrow(h^{\Delta})^{-1}_{FH}(b)\longrightarrow h_{H}(b)\longrightarrow 0.
$$	
By \cite[Proposition 7.2.1]{BSU} these are semiabelian varieties. The dimensional count follows from Lemma \ref{lm:dim_fibers} and the above exact sequence.

Poisson commutativity and linearity of the vectors $X_{h_i}$, $i\,=\,1,\,\cdots,\, N$ follows as in \cite[Proposition 5.12]{BLP}. 

The Hitchin system \eqref{eq:diag_Hitchin} is a maximally abelianizable subsystem as the dimension of the fibers justifies. 
\end{proof}

\begin{remark}\label{rk:relatively_framed}
Given a framed cameral datum, the corresponding Higgs bundle is naturally endowed with a 
framing of the principal bundle and of the Higgs field.
\end{remark}

\begin{remark}\label{rk:max_abelianizable_system}
For general groups $H_x$ one may produce the following maximally abelianizable subsystem. Given $x\in D$,
let $T\subset G$ be a maximal torus, and let $T_{x}\,:=\,T\cap H_x$. Then, one may
consider the stack of cameral data together with a framing, that is, a $T$-equivariant morphism $P|_D\,\longrightarrow\,
\prod_{x\in D_b}T/T_{x}$ which is $T$-equivariant and $W$-equivariant, in the same sense as \eqref{eq:equivar}. The same reasoning as done for $H_x=e$ produces a $\prod_{x\in
D}T/T_{x}$-torsor $(\mathbf{h}_{FH}^{\Delta})^{-1}(b)\,\subset\, \mathbf{h}^{-1}(b)$, that we call
\emph{the stack of framed cameral data} (over $X_b$). On the level of the moduli space, one obtains a
torsor for the group 
$$\left(\prod_{x\in D}T/T_{x}\right)/\left(Z(G)/Z_{H_x}(G)\right)$$ which is
maximal (of dimension $N-\sum_{x\in D}\dim T_{x}+\dim Z_{H_x}(G)$). The fibers are thus semiabelian
varieties of the same dimension as $\mathcal{B}$ if and only if $\dim T_{x}=\dim Z_{H_x}(G)$.
\end{remark}

\section*{Acknowledgements}

Remarks \ref{remP}, \ref{remH} and \ref{remP2} are due to the referee. We are very
grateful to the referee for these and other helpful comments.
The first-named author thanks
Centre de Recherches Math\'ematiques, Montreal, for hospitality. He is partially supported
by a J. C. Bose Fellowship.


\end{document}